\documentclass[11pt]{amsart}
\usepackage{graphicx}
\usepackage{latexsym,amsmath,amsfonts,amscd, amsthm}
\usepackage{changebar}
\usepackage{color}
\usepackage{bm}
\usepackage{tikz}
\usepackage{multirow}
\usetikzlibrary{arrows,backgrounds,snakes}
\usepackage{epsfig,subfigure}
\usepackage{ulem}
\usepackage{mathrsfs}

\newtheoremstyle{plainNoItalics}{}{}{\normalfont}{}{\bfseries}{.}{ }{}
\theoremstyle{plain}
\newtheorem{thm}{Theorem}[section]
\theoremstyle{plainNoItalics}

\newtheorem{defn}[thm]{Definition}
\newtheorem{rem}[thm]{Remark}
\newtheorem{prop}[thm]{Proposition}

\newcommand{\ra}{\rightarrow}
\newcommand{\beq}{\begin{equation}}
\newcommand{\eeq}{\end{equation}}
\newcommand{\beqa}{\begin{eqnarray}}
\newcommand{\eeqa}{\end{eqnarray}}
\newcommand{\bit}{\begin{itemize}}
\newcommand{\eit}{\end{itemize}}
\newcommand{\bedef}{\begin{defn}}
\newcommand{\edefn}{\end{defn}}
\newcommand{\bpro}{\begin{prop}}
\newcommand{\epro}{\end{prop}}

\newcommand{\Dx}{\Delta x}
\newcommand{\Dy}{\Delta y}
\newcommand{\Dt}{\Delta t}

\newcommand{\mO}{{\mathcal O}}

\newcommand{\mI}{{\mathbb  I}}

\newcommand{\eps}{\varepsilon}

\newcommand{\be}{{\bf e}}
\newcommand{\bx}{{\bf x}}

\newcommand{\bu}{{\bf u}}

\newcommand{\bq}{{\bf q}}

\newcommand{\bV}{{\bf V}}

\title[High order AP well-balanced WENO schemes for full Euler with gravity]{High order asymptotic preserving well-balanced finite difference WENO schemes for all Mach full Euler equations with gravity}

		
\begin{document}

\maketitle

\centerline{\scshape Guanlan Huang}
\medskip
{\footnotesize
\centerline{School of Mathematical Sciences, Xiamen University, Xiamen, Fujian, 361005, PR China }
\centerline{glhuang@stu.xmu.edu.cn}
}
	
\medskip

\centerline{\scshape Yulong Xing}
\medskip
{\footnotesize
	\centerline{Department of Mathematics, The Ohio State University, Columbus, OH 43210, USA  }
	\centerline{xing.205@osu.edu}
}

\medskip

\centerline{\scshape Tao Xiong\footnote{Corresponding author.}}
\medskip
{\footnotesize
	\centerline{School of Mathematical Sciences, Fujian Provincial Key Laboratory of Mathematical Modeling} 
	\centerline{and High-Performance Scientific Computing, Xiamen University}
	\centerline{Xiamen, Fujian 361005, PR China}
	\centerline{txiong@xmu.edu.cn}
}

\bigskip
	
	\begin{abstract}
		In this paper, we propose a high order semi-implicit well-balanced finite difference scheme for all Mach Euler equations with a gravitational source term. To obtain the asymptotic preserving property, we start from the conservative form of full compressible Euler equations 
		and add the evolution equation of the perturbation of potential temperature. 
		The resulting system is then split into a (non-stiff) nonlinear low dynamic material wave to be treated explicitly, and (stiff) fast acoustic and gravity waves to be treated implicitly. With the aid of explicit time evolution for the perturbation of potential temperature, we design a novel well-balanced finite difference WENO scheme for the conservative variables, which can be proven to be both asymptotic preserving and asymptotically accurate in the incompressible limit. Extensive numerical experiments were provided to validate these properties.
	\end{abstract}

	
	
	\section{Introduction}
	\label{sec1}
	
	\setcounter{equation}{0}
	\setcounter{figure}{0}
	\setcounter{table}{0}
	
	In this paper, we are interested in the compressible Euler equations with a gravitational source
	\begin{equation}
		\label{FEe1}
		\left\{
		\begin{array}{ll}
			\rho_{t} + \nabla \cdot(\rho \bu) =0, \\ [1mm]
			(\rho \bu)_t+\nabla \cdot(\rho \bu \otimes \bu)+\nabla p =-\rho \nabla \Phi, \\[1mm]
			E_{t}+\nabla \cdot\left( (E+p)\bu\right) = -\rho \bu\cdot\nabla \Phi,
		\end{array}\right.
	\end{equation}
	where $\rho$ is the density, $\bu$ is the velocity, $p$ is the pressure, and $\rho\bu$ is the momentum.
	$\Phi=\Phi(\bx)$ is the gravitational potential which is assumed to be time-independent. $E=\frac{1}{2}\rho|\bu|^2 + \rho e$ is the total non-gravitational energy, with $e$ being the specific internal energy. The system \eqref{FEe1} needs to be closed by providing an equation of state (EOS), which is usually given in the form $e=\mathcal{E}(\rho,p)$. For an ideal gas, the EOS can be written as $e = p/(\gamma -1)/\rho$,
	and the energy $E$ becomes
	\begin{equation}
		\label{E2}
		E= \frac{1}{2}\rho|\bu|^2  + \frac{p}{\gamma-1},
	\end{equation}
	with $\gamma > 1$ being the ratio of specific heat.
	
	On one hand, for hyperbolic systems with a source term, one important feature is that they admit equilibrium state solutions, and well-balanced numerical methods are desirable to exactly preserve such equilibrium states, so that small perturbations around an equilibrium state can be well captured on relatively coarse mesh sizes. In recent years, well-balanced schemes are very attractive for shallow water equations with source terms, see \cite{bermudez1994upwind,kurganov2002solution,audusse2004fast,xing2005high,xing2006high,xing2006new,noelle2006well}, the review papers \cite{X2017,Kurganov18} and the references therein. The compressible Euler equations with a gravitational source \eqref{FEe1} have a zero-velocity hydrostatic equilibrium state \cite{xing2013high} of the form
	\begin{equation}
		\label{equistate}
		\bu = {\bf 0}, \qquad \nabla p = -\rho\nabla \Phi.
	\end{equation}
	Numerically, similar to the shallow water equations, it is essential to develop well-balanced schemes for the Euler equations \eqref{FEe1} to exactly preserve such an equilibrium state, especially for long-time simulations. 
	Many well-balanced schemes have been studied, including finite difference schemes \cite{xing2013high,li2018well}, finite volume schemes \cite{kappeli2014well,Chandrashekar2015,berberich2016general,
		desveaux2016well,li2016high,touma2016well,bermudez2017finite,berberich2019second,grosheintz2019high,klingenberg2019arbitrary}, and discontinuous Galerkin finite element methods \cite{wu2021uniformly,chandrashekar2017well,li2018well2}. These schemes are usually based on explicit time discretizations for compressible flows. 
	
	On the other hand, the Euler equations \eqref{FEe1} have many applications with a wide range of Mach numbers. We can rewrite \eqref{FEe1} into a dimensionless form, by introducing a set of dimensionless variables with some suitable reference values
	\begin{equation}
		\label{DV1}
		\hat{{\bx}} = \frac{{\bx}}{\ell_{ref}},  \,
		\hat{t}     = \frac{t}{t_{ref}},         \,
		\hat{\rho}  = \frac{\rho}{\rho_{ref}},   \,
		\hat{\bu}   = \frac{\bu}{U_{ref}},       \,
		\hat{p}     = \frac{p}{p_{ref}},         \,
		\hat{E}     = \frac{E}{p_{ref}},         \,
		\hat{\Phi}  = \frac{\Phi}{\Phi_{ref}},
	\end{equation}
	where $U_{ref}=\ell_{ref}/t_{ref}$. Under these dimensionless variables, the system \eqref{FEe1} becomes \cite{thomann2020all,birke2021low,benacchio2014blended}
	\begin{subequations}
		\label{FEe2}
		\begin{align}
			& \rho_{t} + \nabla \cdot(\rho \bu) =0, \label{FEe2a} \\
			& (\rho \bu)_t+\nabla \cdot(\rho \bu \otimes \bu)+\frac{1}{\eps^2}\nabla p =-\frac{1}{\eps^2}\rho \nabla\Phi,  \label{FEe2b} \\
			& E_{t}+\nabla \cdot\left( (E+p)\bu\right) = -\rho \bu\cdot\nabla \Phi. \label{FEe2c}
		\end{align}
	\end{subequations}
	Here we drop the hats for those dimensionless variables for clarity. The parameter $\eps = {U_{ref}}/{c_{ref}}$ is a referenced global Mach number, with  $c_{ref} = \sqrt{ {p_{ref}} / {\rho_{ref}}}$ being the acoustic velocity depending on the background flow. The EOS \eqref{E2} becomes
	\begin{equation}
		\label{EOS2}
		E = \frac{1}{2}\eps^2\rho|\bu|^2  + \frac{p}{\gamma-1}.
	\end{equation}
	The dimensionless system \eqref{FEe2} with the EOS \eqref{EOS2} is still hyperbolic, and its eigenvalues along the normal direction ${\bf{n}}$ are 
	\begin{equation}
		\label{eigv}
		\lambda_1 = \bu\cdot{\bf{n}} - c/\eps, \quad \lambda_2 = \bu\cdot{\bf{n}}, \quad \lambda_3 = \bu\cdot{\bf{n}} + c/\eps, 
	\end{equation}
	where $c=\sqrt{\gamma p / \rho}$ is the scaled speed of sound. The model \eqref{FEe2}-\eqref{EOS2} with the global Mach number $\eps$ ranging from $0$ to $\mathcal{O}(1)$ is widely used for all-speed flows \cite{bispen2017asymptotic,thomann2020all}.
	
	For the system \eqref{FEe2}, due to the stiffness of eigenvalues given in \eqref{eigv}, shock capturing schemes with explicit time discretizations are subject to a very strict time step restriction \cite{barsukow2017low}
	\begin{equation} \label{sec1:CFL}
		\Delta t = \text{CFL} \frac{\Delta x}{\max(|\bu|+c/\eps)}\sim \eps \Delta x,
	\end{equation}
	when $\eps$ is small. Here $\Delta t$ is the time step,  $\Delta x$ is the mesh size and CFL denotes the CFL number. Besides, for schemes with numerical viscosities depending on the eigenvalues \eqref{eigv}, e.g. schemes with a Lax-Friedrichs numerical flux, their numerical viscosities would be inversely proportional to $\eps$. As $\eps$ goes to $0$, either the time step is too small, or the numerical viscosities are too large, leading to very inefficient numerical schemes \cite{xing2013high,li2018well}. To avoid that, several semi-implicit schemes have been developed for all-Mach Euler systems \cite{barsukow2017low,kopera2014analysis,thomann2020all,bispen2017asymptotic,benacchio2014blended}, which devote to efficient schemes of easy implementation and uniform time stabilities with respect to $\eps$, and avoid nonlinear iterative solvers as used in the fully implicit schemes. 
	
	In this work, we aim to provide an efficient high order solver for all-Mach full Euler equations with a gravitational source \eqref{FEe2}, which enjoy the benefit of being high order, well-balanced and asymptotic preserving (AP) simultaneously. The design of high order AP method for Euler equations with gravity is a nontrivial task. The target model \eqref{FEe2} is different from all-Mach homogeneous Euler equations considered in \cite{boscarino2022high}, where the stiff acoustic wave only connects to the energy equation, so that only the gradient of pressure in the momentum equation and the flux in the energy equation deserve implicit treatments. All Froude number shallow water equations with a non-flat bottom topography were recently studied in \cite{huang2022high}, where the source term (the gradient of the bottom function) can be merged with the pressure gradient, due to the polytropic EOS $p=p(h)$ with $h$ being the water height. For our model \eqref{FEe2}, the stiff acoustic wave is balanced by the gravity wave, therefore both of them require implicit treatments, which in turn leads to the requirement of implicit treatments for both the density equation \eqref{FEe2a} and the energy equation \eqref{FEe2c}. 
	As a result, one new challenge is that, due to the implicit discretizations for these conservative variables, it is very hard to numerically derive a consistent discretization for the perturbation of the potential temperature $\theta_2$, which appears in the incompressible limit, see \eqref{FEe_LE1}. This is mainly due to the highly nonlinear relationships between conservative and primitive variables, and the nonlinear EOS \eqref{EOS2}. Therefore, the existing approaches in \cite{boscarino2022high,huang2022high} cannot be directly applied.
	
	To address such an issue, one novel contribution of this paper is to add the evolution of the perturbation $\theta_2$ of the potential temperature to \eqref{FEe2} in the numerical discretization. 
	The original energy equation corresponds to the updating of hydrostatic pressure in the compressible regime for conservational purpose, and the added equation of $\theta_2$ extracts the contribution of the energy equation in the scale of $\mathcal{O}(\eps^2)$, and is corresponding to the update of the hydrodynamic pressure $p_2$ in the incompressible regime. The added equation of $\theta_2$ has simplified the treatment of the nonlinear EOS \eqref{EOS2}, which helps us to ensure that the correct asymptotic limit is achieved.
	Then, we split the system \eqref{FEe2} into a (non-stiff) nonlinear low dynamic material wave to be treated explicitly, and (stiff) fast acoustic and gravity waves to be treated implicitly. A high order implicit-explicit (IMEX) method is utilized as the temporal discretization. With the aid of explicit time evolution for the perturbation of potential temperature $\theta_2$, we design a novel well-balanced finite difference WENO scheme for the spatial discretization. The temporal update of the conservative variables $(\rho, \rho \bu, E)$ is through the conservative scheme, and $\theta_2$ is only an auxiliary variable to aid with the design of AP method. The proposed method can be formally proved to be AP and asymptotically accurate (AA). In addition, in the AP and AA analyses, we can show that, as $\eps\ra 0$, \eqref{FEe2c} is a consistent updating of hydrostatic pressure, while $\theta_2$ contributes to the hydrodynamic pressure. Extensive one- and two-dimensional numerical experiments demonstrate the high order accuracy, well-balanced, AP and AA properties of our proposed approach, as well as good performances in both high and low Mach regimes.
	
	The rest of the paper is organized as follows. In Section \ref{sec2}, a low Mach limit of all-Mach Euler equations with gravity is reviewed. AP well-balanced numerical methods are described in Section \ref{sec3}, with AP and AA analyses given in Section \ref{sec4}. Numerical experiments are performed in Section \ref{sec5}, followed by concluding remarks in Section \ref{sec6}. The detailed step-by-step flowchart of the fully discrete high order scheme with a multi-stage IMEX time discretization and a corresponding detailed proof of its AA property are presented as supplementary materials.
	
	\section{Low Mach limit for full Euler equations with gravity}
	\label{sec2}
	
	\setcounter{equation}{0}
	\setcounter{figure}{0}
	\setcounter{table}{0}
	
	In this section, let us briefly review the low Mach limit of Euler equations with gravity \eqref{FEe2}. To derive such a limit, it would be more convenient to start with an equation for the pressure $p$, instead of \eqref{FEe2c} for the total energy. By utilizing the EOS \eqref{EOS2}, \eqref{FEe2c} can be replaced by
	\begin{equation}
		\label{P1}
		p_t + \bu \cdot \nabla p + \gamma p \nabla \cdot \bu = 0.
	\end{equation}
	If we further define a potential temperature $\theta$ from $p=(\rho\theta)^{\gamma}$, it yields
	\begin{equation}
		\label{EC1}
		(\rho\theta)_t + \nabla\cdot(\rho\theta\bu) = 0,
	\end{equation}
	which reduces to (after combined with \eqref{FEe2a})
	\begin{equation}
		\label{eq:theta}
		\theta_t + \bu\cdot\nabla\theta = 0,
	\end{equation}
	namely, the potential temperature $\theta$ satisfies a simple transport equation with the velocity $\bu$. Now the model \eqref{FEe2} can be rewritten as
	\begin{subequations}
		\label{FEe_ND}
		\begin{align}
			& \rho_{t} + \nabla \cdot(\rho \bu) =0, \label{FEe_NDa} \\
			& (\rho \bu)_t+\nabla \cdot(\rho \bu \otimes \bu)+\frac{1}{\eps^2}\nabla p =-\frac{1}{\eps^2}\rho \nabla\Phi, \label{FEe_NDb} \\
			& \theta_t+\bu\cdot\nabla \theta =0. \label{FEe_NDc}
		\end{align}
	\end{subequations}
	Let us perform a Chapman-Enskog expansion in the following form \cite{bispen2017asymptotic}
	\begin{equation}
		\label{Exp2}
		\left\{
		\begin{array}{ll}
			\rho(\bx,t) = \rho_0(\bx,t) +\eps^2 \rho_2(\bx,t) + \cdots,\\ [1mm]
			\bu(\bx,t) = \bu_0(\bx,t)+ \eps \bu_1(\bx,t) + \cdots,\\ [1mm]
			p(\bx,t) = p_0(\bx,t)+ \eps^2 p_2(\bx,t) + \cdots,\\ [1mm]
			\theta(\bx,t) = \theta_0(\bx,t)+ \eps^2 \theta_2(\bx,t)+ \cdots.
		\end{array}\right.
	\end{equation}
	By substituting them into \eqref{FEe_NDb}, we can collect the leading order $\mathcal{O}(1/\eps^2)$ terms
	\begin{equation}
		\label{ES2}
		\nabla p_0 = -\rho_0 \nabla\Phi,
	\end{equation}
	where $p_0$ is the hydrostatic pressure. The $\mathcal{O}(1)$ terms of \eqref{FEe_NDb} yield
	\begin{equation}
		\label{m2}
		(\rho_0 \bu_0)_t+\nabla \cdot(\rho_0 \bu_0 \otimes \bu_0)+\nabla p_2 =-\rho_2 \nabla\Phi,
	\end{equation}
	where $p_2$ corresponds to the hydrodynamic pressure.
	
	To obtain a closed system for the low Mach limit, a further assumption is needed due to the extra source term $-\rho_2\nabla\Phi$ appearing in \eqref{m2}, which is different from the homogeneous all Mach full Euler system without gravity \cite{boscarino2022high}. One commonly used assumption \cite{klein2009asymptotics,bispen2017asymptotic} is a constant background potential temperature $\theta_0$ with hydrostatic background states $\rho_0(\bx)$ and $p_0(\bx)$ satisfying \eqref{ES2}. With such assumption, from \eqref{FEe_NDa}, \eqref{m2}, \eqref{FEe_NDc} and the expansion \eqref{Exp2}, we obtain the following limiting incompressible equations
	\begin{equation}
		\label{FEe_LE1}
		\left\{
		\begin{array}{ll}
			\nabla \cdot\left(\rho_0 \bu_0\right) =0, \\ [2mm]
			\left(\rho_0 \bu_0\right)_t+\nabla \cdot(\rho_0 \bu_0 \otimes \bu_0)+\nabla p_2 =-\rho_2 \nabla\Phi, \\[2mm]
			(\theta_2)_t + \bu_0\cdot\nabla \theta_2 = 0,
		\end{array}\right.
	\end{equation}
	which is known as the Ogura and Phillips model \cite{ogura1962scale}.
	The model \eqref{FEe_LE1} contains four unknowns $\bu_0$, $\rho_2$, $p_2$, $\theta_2$, and is closed by the relation $p=(\rho\,\theta)^\gamma$, which yields
	\begin{equation}
		\label{C1}
		p_2 = \lim_{\eps\rightarrow 0} \frac{p-p_0}{\eps^2} =\lim_{\eps\rightarrow 0} \frac{(\rho\,\theta)^\gamma-(\rho_0\,\theta_0)^\gamma}{\eps^2}=\frac{\gamma p_0}{\rho_0\theta_0}(\rho_0\,\theta_2+\theta_0\,\rho_2),
	\end{equation}
	where the Chapman-Enskog expansion \eqref{Exp2} is used. 
	
	As compared to the homogeneous all Mach full Euler system without gravity \cite{boscarino2022high}, an extra equation of $\theta_2$ exists in the incompressible system \eqref{FEe_LE1}. Due to the existence of the source term in the momentum equation of \eqref{FEe_LE1}, the limit model \eqref{FEe_LE1} requires \eqref{C1} to close the system, which will lead to additional numerical challenges as described in the next section.
	
	\begin{rem}
		Another model assumes a stratified background potential temperature $\theta_0(\bx)$ with $\nabla \theta_0(\bx) = \mathcal{O}(\eps^2)$, which is typically used in a nearly incompressible regime. Similarly if we have hydrostatic background states $\rho_0(\bx)$ and $p_0(\bx)$ satisfying \eqref{ES2}, the nearly incompressible system is given as
		\begin{equation}
			\label{FEe_LE2}
			\left\{
			\begin{array}{ll}
				\nabla \cdot\left(\rho_0 \bu_0\right) =0, \\ [2mm]
				\left(\rho_0 \bu_0\right)_t+\nabla \cdot(\rho_0 \bu_0 \otimes \bu_0)+\nabla p_2 =-\rho_2 \nabla\Phi, \\[2mm]
				(\theta_2)_t + \bu_0\cdot\nabla \theta_2 + \frac{1}{\eps^2}\bu_0\cdot\nabla\theta_0 = 0.
			\end{array}\right.
		\end{equation}
		This is known as the Bannon's  anelastic model \cite{bannon1996anelastic}. As $\eps\rightarrow 0$, $\theta_0(\bx)$ converges to a constant and \eqref{FEe_LE2} becomes \eqref{FEe_LE1}.	
	\end{rem}

	\section{Numerical scheme}
	\label{sec3}
	\setcounter{equation}{0}
	\setcounter{figure}{0}
	\setcounter{table}{0}
	
	In this section, we propose a class of semi-implicit high order finite difference weighted essentially non-oscillatory (WENO) schemes, with both AP and well-balanced properties, for the all-Mach Euler system with gravity \eqref{FEe2}. 
	A novel new framework is presented to design a well-balanced high order semi-implicit scheme, that is AP and AA in the incompressible limit, and is conservative and robust in the compressible regime. 
	
	\subsection{Model problem and semi-implicit splitting}
	For the model \eqref{FEe2}-\eqref{EOS2}, the gravitational source appears both in the second momentum and third energy equations. It couples all the conservative unknowns together, which is different from the source term due to a non-flat bottom topography in the shallow water equations \cite{huang2022high}. Compared to all Mach homogeneous Euler equations without gravity in \cite{boscarino2022high}, this introduces two extra variables $\rho_2$ and $\theta_2$ in the limiting equations \eqref{FEe_LE1}. As a result, a fully nonlinear coupling of all conservative variables exists, and it is highly nontrivial to extract an elliptic equation for the extra variables $\rho_2$ or $p_2$ from \eqref{FEe2} through the nonlinear EOS \eqref{EOS2}. Note that such elliptic equation is very important to ensure a correct low Mach limit \eqref{FEe_LE1}. Therefore, the frameworks developed in \cite{huang2022high,boscarino2022high} cannot be directly applied to \eqref{FEe2}. 
	
	In order to address the difficulty to ensure a correct asymptotic limit from \eqref{FEe2}, one novel idea of this paper is to include the evolution of perturbation of the potential temperature $\theta_2$ to the system 
	\begin{equation}
		\label{S3_1}
		\left\{
		\begin{array}{ll}
			\rho_t + \nabla\cdot(\rho\bu) = 0,\\[2mm]
			(\rho \bu)_t+\nabla \cdot(\rho \bu \otimes \bu )+ \frac{1}{\eps^2}\nabla p = - \frac{1}{\eps^2}\rho\nabla\Phi, \\[2mm]
			E_t + \nabla \cdot \left( \left(E+p\right)\bu \right) = -\rho\bu \cdot\nabla\Phi,\\[2mm]
			(\theta_2)_t + \bu \cdot \nabla \theta_2 + \frac{1}{\eps^2}\bu \cdot\nabla\theta_0 = 0.
		\end{array}
		\right.
	\end{equation}
	Note that the equation of $\theta_2$ is taken from \eqref{FEe_NDc} by assuming an expansion $\theta(\bx,t)=\theta_0(\bx)+\eps^2\theta_2(\bx,t)$, mimicking the one in \eqref{Exp2}. Here $\theta_0(\bx)$ is assumed to be a pre-determined time independent background potential temperature and $\nabla\theta_0(\bx)=\mathcal{O}(\eps^2)$. This equation can ensure the right incompressible limit \eqref{FEe_LE1}, and more generally is consistent with \eqref{FEe_LE2} in the nearly incompressible regime.
	The advantage of \eqref{S3_1}, compared with \eqref{FEe2}, is that the third energy equation corresponds to the updating of hydrostatic pressure $p_0$ (in the compressible regime) for conservational purpose, while the fourth equation of $\theta_2$ extracts the contribution of the energy equation in the scale of $\mathcal{O}(\eps^2)$, and is corresponding to the hydrodynamic pressure $p_2$ which plays the role as a Lagrangian multiplier in the incompressible regime. The fourth equation of $\theta_2$, which is derived from \eqref{FEe_NDc}, has simplified the treatment of the nonlinear EOS \eqref{EOS2} as compared to the third energy equation, and this helps us to easily form an elliptic equation for $\rho_2$ as described in the design of our scheme below.
	
	We first follow the idea in \cite{haack2012all,tang2012} to split the system \eqref{S3_1} into a stiff part and a non-stiff component
	\begin{equation}
		\label{split}
		\frac{dU}{dt} = -\nabla \cdot\mathcal{F}_E -\nabla \cdot\mathcal{F}_I  - \mathcal{S}_E - \mathcal{S}_I,
	\end{equation}
	where $U=(\rho,\rho\bu,E,\theta_2)^T$ and
	\begin{subequations}
		\begin{align}
			&\nabla\cdot\mathcal{F}_E
			= \alpha\Big(
			\nabla\cdot (\rho\bu)_E,
			\nabla\cdot(\rho\bu\otimes\bu + p\,\mI)_E,
			\nabla\cdot((E+p)\bu)_E,
			0
			\Big)^T
			\\ &
			+ 
			\Big(
			0, (1-\alpha)\,\nabla\cdot(\rho\bu\otimes\bu)_E+ \alpha\frac{1-\eps^2}{\eps^2}\nabla p_E,
			0,
			\bu_E\cdot\nabla\theta_{2,E} + \frac{1}{\eps^2}\bu_E\cdot\nabla\theta_0
			\Big)^T, \notag
			\\
			&\nabla\cdot\mathcal{F}_I
			= (1-\alpha)\left(
			\nabla\cdot(\rho\bu)_I,
			\frac{1}{\eps^2}\nabla p_I,
			\nabla\cdot((E+p)\bu)_I,
			0
			\right)^T,
			\\
			&\mathcal{S}_E
			=
			\alpha\left(
			0,
			\frac{1}{\eps^2} \rho_E\nabla\Phi,
			(\rho\bu)_E\cdot\nabla\Phi,
			0
			\right)^T,
			\\
			& \mathcal{S}_I
			= (1-\alpha)
			\left(
			0,
			\frac{1}{\eps^2} \rho_I\nabla\Phi,
			(\rho\bu)_I\cdot\nabla\Phi,
			0
			\right)^T.
		\end{align}
	\end{subequations}
	The subscripts $E$ and $I$ represent non-stiff and stiff parts of the flux function $\mathcal{F}$ and the source term $S$, which will be discretized explicitly and implicitly, respectively. In the splitting, we take the splitting parameter $\alpha = \min(\eps^2, 1)$, that is, $\alpha=\eps^2$ when $\eps\le 1$ and \eqref{split} turns to be fully explicit if $\eps\ge1$. In the following, we only consider $\alpha=\eps^2$ with $\eps \le 1$ for ease of presentation.
	
	\subsection{First order IMEX scheme}\label{sec3.2}
	We start with presenting the first order IMEX AP time discretization for \eqref{S3_1}. We focus on time discretization only and keep space to be continuous. The spatial discretization will be discussed afterward. Following the splitting in \eqref{split}, the semi-discrete (in time) scheme takes the form
	\begin{subequations}
		\label{IMEX1}
		\begin{align}
			&\frac{\rho^{n+1} - \rho^n}{\Dt} + \eps^2\,\nabla\cdot(\rho\bu)^{n} + (1-\eps^2)\,\nabla\cdot(\rho\bu)^{n+1} = 0, \label{IMEX1:a} \\
			&\frac{(\rho \bu)^{n+1} - (\rho \bu)^n}{\Dt}+\nabla \cdot(\rho \bu \otimes \bu + p\,\mI)^n+\frac{1-\eps^2}{\eps^2}\nabla p^{n+1}  \label{IMEX1:b} \\ 
			& \hspace{3.5cm} =- \rho^n\nabla\Phi - \frac{1-\eps^2}{\eps^2}\rho^{n+1} \nabla\Phi,   \notag \\
			&\frac{E^{n+1} - E^n}{\Dt} + \eps^2\,\nabla\cdot \left((E+p)\bu\right)^{n}
			+(1-\eps^2)\,\nabla\cdot \left((E+p)\bu\right)^{n+1}     \label{IMEX1:c} \\ 
			& \hspace{3.5cm}
			= -\eps^2\,\left(\rho\bu\right)^{n} \cdot\nabla\Phi -(1-\eps^2)\,\left(\rho\bu\right)^{n+1}\cdot \nabla\Phi,  \notag \\
			&\frac{\theta_2^{n+1} - \theta_2^n}{\Dt}+ \bu^n\cdot\nabla\theta_2^n +\frac{1}{\eps^2}\bu^n\cdot\nabla\theta_0 = 0. \label{IMEX1:d}
		\end{align}
	\end{subequations}
	Introducing the variables 
	\begin{subequations}
		\label{Ustar}
		\begin{align}
			&\widetilde{\rho} = \rho^n - \Dt\,\eps^2\,\nabla\cdot(\rho\bu)^{n}, \label{Ustara}\\
			&\widetilde{\rho\bu}=  (\rho \bu)^n  - \Dt\Big(\nabla \cdot(\rho \bu \otimes \bu + p\mI)^n + \rho^n\nabla\Phi\Big), \label{Ustarb}\\
			&\widetilde{E}=E^n-\Dt\,\eps^2\,\Big( \nabla\cdot \left((E+p)\bu\right)^{n}
			+\,\left(\rho\bu\right)^{n} \cdot\nabla\Phi\Big), \label{Ustarc}
		\end{align}
	\end{subequations}	
	we can rewrite \eqref{IMEX1} as
	\begin{subequations}
		\label{S3_8}
		\begin{align}
			&\rho^{n+1} = \widetilde{\rho} -\Dt\, (1-\eps^2)\,\nabla \cdot (\rho\bu)^{n+1},
			\label{S3_8_1} \\
			&(\rho \bu)^{n+1}= \widetilde{\rho \bu}
			- \Dt\,\frac{1-\eps^2}{\eps^2} \left(\nabla p^{n+1} + \rho^{n+1} \nabla\Phi\right),
			\label{S3_8_2} \\
			&E^{n+1} = \widetilde{E} -\Dt\,(1-\eps^2)\,\left(\nabla\cdot \left((E+p)\bu\right)^{n+1}  +\,\left(\rho\bu\right)^{n+1}\cdot \nabla\Phi\right), \label{S3_8_3} \\
			&\theta_2^{n+1} = \theta_2^n - \Dt\, \left(\bu^n\cdot\nabla\theta_2^n\right) - \Dt\,\left(\frac{1}{\eps^2}\bu^n\cdot\nabla\theta_0\right),
			\label{theta2} 
		\end{align}
	\end{subequations}
	which will be solved with the following three steps: 
	\bit
	\item {\bf Step 1}: 
	Combining \eqref{S3_8_1} and \eqref{S3_8_2} leads to
	\begin{equation}
		\label{S3_9}
		\rho^{n+1} = \widetilde{\rho}- (1-\eps^2)\,\Dt\,\nabla \cdot (\widetilde{\rho\bu}) + \Dt^2\,\frac{(1-\eps^2)^2}{\eps^2}\nabla\cdot\left(\nabla p^{n+1} +  \rho^{n+1}\nabla\Phi\right),
	\end{equation}
	We now introduce $p_2$ and $\rho_2$ from
	\begin{equation}
		\label{pr}
		p^{n+1} = p_0(\bx) + \eps^2\,p_2^{n+1}, \quad \rho^{n+1} = \rho_0(\bx) + \eps^2\,\rho_2^{n+1},
	\end{equation}
	according to \eqref{Exp2}, where $p_0(\bx)$ and $\rho_0(\bx)$ satisfy $\nabla p_0 = - \rho_0\nabla\Phi$ as given in \eqref{ES2}. Therefore \eqref{S3_9} becomes
	\begin{equation}
		\label{eqr2p2}
		\eps^2\rho_2^{n+1} = \widetilde{\rho}- (1-\eps^2)\,\Dt\,\nabla \cdot (\widetilde{\rho\bu})+ \Dt^2(1-\eps^2)^2\nabla\cdot\left(\nabla p_2^{n+1} +  \rho_2^{n+1}\nabla\Phi\right)-\rho_0,
	\end{equation}
	which cannot be solved due to two unknowns $p^{n+1}_2$ and $\rho_2^{n+1}$. From the relation $p=(\rho\theta)^\gamma$, we have
	\begin{equation}
		\label{p2a}
		p_2^{n+1} = (p-p_0)/\eps^2
		= \left[\Big((\rho_0+\eps^2\,\rho_2^{n+1})(\theta_0+\eps^2\,\theta_2^{n+1})\Big)^\gamma-p_0\right]/\eps^2.
	\end{equation}
	Substituting it into \eqref{eqr2p2} leads to a nonlinear system for $\rho_2^{n+1}$, which can be solved iteratively. To simplify the computation, we approximate \eqref{p2a} with a linearization up to an approximation error $\mathcal{O}(\eps^2)$,
	\begin{equation}
		\label{p2}
		p_2^{n+1}
		=\frac{\gamma p_0}{\rho_0\theta_0}(\rho_0\theta_2^{n+1} + \rho^{n+1}_2\,\theta^{n+1}),
	\end{equation}
	where we denote $\theta^{n+1} = \theta_0 + \eps^2\,\theta_2^{n+1}$.
	Substituting \eqref{p2} into \eqref{eqr2p2} yields
	\begin{equation}
		\label{S3_10}
		\eps^2\rho_2^{n+1} = \widetilde{\widetilde{\rho}}+ \Dt^2(1-\eps^2)^2\nabla\cdot\left(\nabla \left(\frac{\gamma p_0}{\rho_0\theta_0}\theta^{n+1}\,\rho_2^{n+1}\right) +  \rho_2^{n+1}\nabla\Phi\right), 
	\end{equation}
	where
	\begin{equation}
		\widetilde{\widetilde{\rho}}= \widetilde{\rho} - (1-\eps^2)\,\Dt\,\nabla \cdot (\widetilde{\rho\bu}) + \Dt^2(1-\eps^2)^2\Delta\left( \frac{\gamma p_0}{\theta_0}\theta^{n+1}_2\right)- \rho_0.
	\end{equation}
	The purpose of adding the fourth equation on $\theta_2$ in \eqref{S3_1} is to evaluate $\theta^{n+1}_2$ from \eqref{theta2} explicitly. With such information, $\theta^{n+1}$ can be computed and \eqref{S3_10} is only a linear system for $\rho_2^{n+1}$. Otherwise, without \eqref{theta2}, \eqref{S3_10} will be coupled to \eqref{S3_8_3}, which is very complicated to solve for $\rho_2$ and $\theta_2$. 
	After obtaining $\rho^{n+1}_2$, $p_2^{n+1}$ is available from \eqref{p2}.
	
	\item {\bf Step 2}:  
	$(\rho\bu)^{n+1}$ directly follows from \eqref{S3_8_2} and \eqref{pr}.
	In order to ensure exact mass and energy conservation, especially in the compressible regime, we further update $\rho^{n+1}$ from \eqref{S3_8_1}, and finally $E^{n+1}$ from \eqref{S3_8_3}.
	We note that in \eqref{S3_8_3}, the implicit term $((E+p)\bu)^{n+1}$ is computed based on
	\begin{equation}
		((E+p)\bu)^{n+1}=H^{n+1}\,(\rho\bu)^{n+1}, 
	\end{equation}
	where the enthalpy $H^{n+1}$ is given by
	\begin{equation}
		\label{H}
		H^{n+1}=\left(\frac{E+p}{\rho}\right)^{n+1}=\frac{\gamma}{\gamma-1}\frac{p^{n+1}}{\rho^{n+1}}+\frac12\eps^2\frac{|(\rho\bu)^{n+1}|^2}{(\rho^{n+1})^2},
	\end{equation}
	following the EOS \eqref{EOS2}. To compute $H^{n+1}$, we use $\rho^{n+1}$ and $p^{n+1}$ from \eqref{pr} available after step 1. In this way, $(\rho\bu)^{n+1}$, $\rho^{n+1}$ and $E^{n+1}$ are updated in a conservative form and explicitly via a sequential way.
	\eit
	
	\begin{rem}
		\label{rem31}
		In the momentum equation \eqref{IMEX1:b}, the stiff gravitational source $\rho\nabla\Phi$ which depends on the unknown variable $\rho$ needs to be discretized implicitly. Correspondingly, we need an implicit discretization for the density equation \eqref{IMEX1:a} in the low Mach regime. On the other hand, the implicit discretization of $\nabla p$ needs to be coupled with an implicit discretization of the energy equation \eqref{IMEX1:c}. In this way, the mass, momentum and energy equations in \eqref{IMEX1} all have components needing implicit discretizations, and uniform time stability independent of $\eps$ can be obtained. This can be viewed as, in some sense, a combination of two approaches for the isentropic and full Euler equations \cite{boscarino2019high,boscarino2022high}. We refer to \cite{boscarino2019high} for the Fourier analysis with respect to the rationale behind these implicit discretizations.
	\end{rem}

	\begin{rem} 
		In \cite{bispen2017asymptotic}, the authors used \eqref{EC1} instead of \eqref{FEe2c}, and designed an AP scheme in the low Mach regime. $\rho\theta$ in \eqref{EC1} is not a physically conservative variable. In this work, we keep the energy equation \eqref{FEe2c}, which is important for shock capturing purpose in the high Mach regime. Besides, a numerical discretization of \eqref{EC1} does not straightforwardly result in a consistent discretization of $\theta_2$ in the limiting equation \eqref{FEe_LE1} as $\eps\rightarrow 0$. It requires $\nabla\cdot(\rho\bu)=0$ to be satisfied at the discrete level as $\eps\rightarrow 0$, which is not an easy task.
	\end{rem}
	
	\subsection{High order IMEX Runge-Kutta (IMEX-RK) scheme}\label{sec3.3}
	We now generalize the first-order semi-implicit scheme \eqref{IMEX1} to high-order, by utilizing a high order IMEX-RK scheme proposed in \cite{boscarino2016high},
	which has already been used in \cite{boscarino2019high,boscarino2022high}.
	We first denote the split system \eqref{split} as 
	\begin{equation}
		\label{ho}
		U_t = \mathcal{H} (U_E,U_I), \qquad U(t_0) = U_0.
	\end{equation}
	where $\mathcal{H}(U_E,U_I) = -\nabla \cdot\mathcal{F}_E-\nabla \cdot\mathcal{F}_I-\mathcal{S}_E - \mathcal{S}_I$, with $U_E=(\rho_E,(\rho\bu)_E,E_E,\theta_{2,E})^T$ and $U_I=(\rho_I,(\rho\bu)_I,E_I,\theta_{2,I})^T$ representing the terms that will be discretized explicitly and implicitly. 
	The solutions $U_E$ and $U_I$ will be updated separately, that is
	\begin{equation}\label{PartitionedSyst}
		\left\{
		\begin{array}{l}
			(U_E)_t =  \mathcal{H}(U_E,U_I), \\ [2mm]
			(U_I)_t =  \mathcal{H}(U_E,U_I),
		\end{array}
		\right.
	\end{equation}
	with the same initial condition
	\begin{equation}
		\label{PartitionedSyst_IF}
		U_E(t_0) = U_I(t_0) = U_0.
	\end{equation}
	Following \cite{boscarino2019high,boscarino2022high}, we apply an IMEX-RK scheme to \eqref{PartitionedSyst} with a double Butcher tableau \cite{butcher2016}
	\begin{equation}\label{DBT}
		\begin{array}{c|c}
			\tilde{c} & \tilde{A}\\
			\hline
			\vspace{-0.25cm}
			\\
			& \tilde{b}^T \end{array}, \ \ \ \ \
		\begin{array}{c|c}
			{c} & {A}\\
			\hline
			\vspace{-0.25cm}
			\\
			& {b^T} \end{array},
	\end{equation}
	where $\tilde{A} = (\tilde{a}_{ij})$ is an $s \times s$ matrix for an explicit scheme with $\tilde{a}_{ij}=0$ for $j \geq i$, and $A = ({a}_{ij})$ is an  $s \times s$ matrix for an implicit scheme. A diagonally implicit RK (DIRK) scheme with $a_{ij}=0$ for $j > i$ is used, which is simple and efficient for implicit discretization. Other vectors are $\tilde{c}=(\tilde{c}_1,...,\tilde{c}_s)^T$, $\tilde{b}=(\tilde{b}_1,...,\tilde{b}_s)^T$, and $c=(c_1,...,c_s)^T$, $b=(b_1,...,b_s)^T$, where $\tilde{c}$ and $c$ satisfy the relation
	$\tilde{c}_i = \sum_{j=1}^{i-1} \tilde a_{ij}$, $c_i = \sum_{j=1}^{i} a_{ij}$.
	For the high order IMEX-RK scheme, as discussed in \cite{boscarino2019high,boscarino2022high},
	a stiffly accurate property with $b^T=\be_s^TA$ where  $\be_s=(0,\cdots,0,1)^T$ for the implicit part is preferred, which can yield the AA property of the scheme.
	A sketched procedure of the IMEX-(DI)RK scheme is given as follows:
	\begin{enumerate}
		\item Set $U_E^{(0)}=U_I^{(0)}=U^n$. For each inner stage $i = 1  \text{ to }  s$, 
		\begin{itemize}
			\item Update the explicit solution $U_E^{(i)}$
			\begin{equation}
				\label{S3_E7}
				U_E^{(i)} = U^n + \Delta t\sum^{i-1}_{j=1}\tilde{a}_{ij}\mathcal{H}(U_E^{(j)},U_I^{(j)});
			\end{equation}
			\item Compute $U_\star^{(i)}$ based on known solutions at previous stages
			\begin{equation}
				\label{S3_E8}
				U_\star^{(i)} = U^n + \Delta t\sum^{i-1}_{j=1}a_{ij}\mathcal{H}(U_E^{(j)},U_I^{(j)}),
			\end{equation}
			then update the implicit solution $U_I^{(i)}$ via
			\begin{equation}
				\label{S3_E9}
				U_I^{(i)} = U_\star^{(i)} + \Delta ta_{ii}\mathcal{H}(U_E^{(i)},U_I^{(i)});
			\end{equation}
		\end{itemize}
		\item Update the solution $U^{n+1}$ at the next time level $t^{n+1}$ by
		\begin{equation}
			\label{S3_E15}
			U^{n+1}  = U^n + \Delta t \sum_{i=1}^s b_i \mathcal{H}(U_E^{(i)},U_I^{(i)}).
		\end{equation}
		
		
	\end{enumerate}

	\subsection{Fully discrete scheme with high order spatial discretizations}
	\label{sec3.4}
	In this part, we introduce high order WENO spatial discretizations for \eqref{split}. We expect the following desirable properties from the spatial discretization. Firstly, it should be suitable for all Mach numbers, that is, for $\eps$ ranging from $0$ to $\mathcal{O}(1)$. Secondly, it can preserve the well-balanced property, namely, the zero-velocity steady state solution \eqref{equistate} can be maintained exactly on the discrete level. To achieve both purposes, we adopt the spatial discretizations proposed in \cite{boscarino2022high}, and well-balanced reconstruction strategies from \cite{xing2013high,li2018well}.
	
	Without loss of generality, we present \eqref{split} in the two-dimensional spatial setting. Taking $\alpha=\eps^2$, we denote $U = (\rho,\rho u, \rho v,E,\theta_2)^T$ and
	\begin{subequations} \label{F2d}
		\begin{align}
			&\nabla \cdot\mathcal{F}_E  = \partial_x\mathcal{F}_E ^x + \partial_y\mathcal{F}_E ^y, \quad \mathcal{S}_E = \mathcal{S}_E^x + \mathcal{S}_E^y, \\
			&\nabla \cdot\mathcal{F}_I  = (1-\eps^2)(\partial_x\mathcal{F}_I ^x + \partial_y\mathcal{F}_I ^y), \quad \mathcal{S}_I = (1-\eps^2)(\mathcal{S}_I^x + \mathcal{S}_I^y), 
		\end{align}
	\end{subequations}
	with
	\begin{subequations}
		\label{Fx2d}
		\begin{align}
			&\partial_x\mathcal{F}_E^x = \eps^2\Big(\partial_x(\rho u), \,\partial_x(\rho u^2 + p), \,\partial_x(\rho uv),\,\partial_x((E+p)u), \,0\Big)_E^T \label{FEx} \\
			&\hspace{1.3cm}+(1-\eps^2)\Big(0,\,\partial_x(\rho u^2 + p), \,\partial_x(\rho uv),\,0,\,0\Big)_E^T \notag \\
			&\hspace{1.3cm}+\Big(0,\,0, \,0,\,0,\,u\partial_x\theta_{2}+\frac{1}{\eps^2}u\partial_x\theta_0\Big)_E^T
			\notag \\
			&\partial_y\mathcal{F}_E^y = \eps^2\Big(\partial_y(\rho v),\,\partial_y(\rho uv), \,\partial_y(\rho v^2 + p),\,\partial_y((E+p)v), \,0\Big)_E^T \label{FEy}\\
			&\hspace{1.3cm}+(1-\eps^2)\Big(0,\,\partial_y(\rho uv), \,\partial_y(\rho v^2 + p),\,0, \,0\Big)_E^T \notag \\
			&\hspace{1.3cm}+\Big(0,\,0,\,0,\,0, \,v\partial_y\theta_{2}+\frac{1}{\eps^2}v\partial_y\theta_0\Big)_E^T, \notag \\
			&\partial_x\mathcal{F}_I^x = \Big(\partial_x(\rho u), \,\partial_x p_{2}, \,0,\,\partial_x((E+p)u), \,0\Big)_I^T, \label{FIx}
			\\
			&\partial_y\mathcal{F}_I^y = \Big(\partial_y(\rho v),\,0, \,\partial_yp_{2},\,\partial_y((E+p)v), \,0\Big)_I^T, \label{FIy}
		\end{align}
	\end{subequations}
	and
	\begin{subequations}
		\label{s2d}
		\begin{align}
			&\mathcal{S}_E^x = \Big(0,\rho\,\partial_x\Phi,0,\eps^2\rho u\partial_x\Phi,0\Big)_E^T,
			\quad \mathcal{S}_E^y = \Big(0,0,\rho\,\partial_y\Phi,\eps^2\rho v\partial_y\Phi,0\Big)_E^T, \label{sE}\\
			&\mathcal{S}_I^x = \Big(0,\rho_2\,\partial_x\Phi,0,\rho u\partial_x\Phi,0\Big)_I^T,
			\quad \mathcal{S}_I^y = \Big(0,0,\rho_2\,\partial_y\Phi,\rho v\partial_y\Phi,0\Big)_I^T. \label{sI}
		\end{align}
	\end{subequations}
	Here \eqref{pr} with $\nabla p_0=-\rho_0\nabla\Phi$ is used in \eqref{FIx}, \eqref{FIy} and \eqref{sI}, by replacing $p$ and $\rho$ with $p_2$ and $\rho_2$ respectively. 
	
	For a high order finite difference discretization, we take a uniform cartesian mesh with mesh sizes $\Dx$ and $\Dy$ along $x$ and $y$ directions respectively. The grid points are
	$(x_i,y_j)$ for $i=1,2,\dots,N_x$, $j=1,2,\dots,N_y$. We denote $w_{ij}=w(x_i,y_j)$ for short with any variable $w$, and the interface values are denoted as $w_{i\pm \frac{1}{2},j}$ and $w_{i,j\pm \frac{1}{2}}$ respectively.
	
	We now briefly describe our spatial discretizations. Denoting $\bq=\rho\bu$, the high order WENO approximation of the spatial operator $\mathcal{H}(U_E,U_I)$ defined in \eqref{ho} is given by
	\begin{equation}
		\label{hoh2}
		\mathcal{H}_{WENO}(U_E,U_I)=-\left(
		\begin{aligned}
			&\eps^2\nabla_{CW}\cdot \bq_E
			+(1 -\eps^2)\nabla_{W}\cdot \bq_I
			\\
			&\nabla^{WB}_{CW}\cdot \left(\frac{\bq \otimes \bq}{\rho} + p\, \mI\right)_E  
			-\frac{\rho_E}{\rho_0}\nabla^{WB}_{CW}p_0\\&\hspace{1.2cm} +(1-\eps^2)\left(\nabla_{W}p_{2,I} + \rho_{2,I}\nabla\Phi \right) 
			\\
			&\eps^2\left(\nabla_{CW}\cdot(H\bq)_E
			+ \bq_E\cdot \nabla\Phi\right)  \\&\hspace{1.2cm}
			+(1 -\eps^2)\left(\nabla_{W}\cdot(H\bq)_I +  \bq_I\cdot \nabla\Phi  \right) 
			\\
			&\bu_E^{(j)}\cdot \nabla_{UW} \theta_{2,E} +\frac{1}{\eps^2}\bu_E\nabla\theta_0
		\end{aligned}
		\right),
	\end{equation}
	where the WENO operators $\nabla_{CW}$, $\nabla_{W}$, $\nabla_{UW}$ and $\nabla^{WB}_{CW}$ will be defined below.
	The fully discrete scheme can be obtained by combining the spatial discretization \eqref{hoh2} with the high order IMEX temporal discretization outlined in \eqref{S3_E7}-\eqref{S3_E15}. Note that this involves an implicit step in \eqref{S3_E9}. Following the approach to solve the first order IMEX method \eqref{S3_8} in Section \ref{sec3.2}, we can convert it into an elliptic equation and then update the solution in a sequential way. For the high order scheme in \eqref{S3_E9}, we form the following elliptic equation for $\rho_2$
	\begin{equation}
		\label{rho2}
		\eps^2\rho_{2,I}^{(i)} = \rho_{\star\star\star}^{(i)}
		+ (\Dt a_{ii})^2
		(1 -\eps^2)^2\left(\Delta\Big(\frac{\gamma p_0}{\theta_0}\,\theta_{2,E}^{(i)} + \frac{\gamma p_0}{\rho_0\theta_0} \theta^{(i)}\,\rho_{2,I}^{(i)}\Big)  + \nabla\cdot\Big(\rho_{2,I}^{(i)}\nabla\Phi\Big) \right),
	\end{equation}
	with
	\begin{equation}
		\rho_{\star\star\star}^{(i)} = \rho_{\star\star}^{(i)} - \Dt a_{ii}
		(1 -\eps^2)\nabla_W \cdot \bq_{\star\star}^{(i)} - \rho_0 ,
	\end{equation}		
	where 
	\begin{subequations}
		\begin{align}
			&\rho_{\star\star}^{(i)}=\rho_{\star}^{(i)} - \Dt a_{ii}\eps^2\nabla_{CW}\cdot \bq_E^{(i)},
			\\
			&\bq_{\star\star}^{(i)} = \bq_{\star}^{(i)}   - \Dt a_{ii}
			\left(\nabla^{WB}_{CW}\cdot \left(\frac{\bq \otimes \bq}{\rho} + p \mI\right)_E^{(i)} -
			\frac{\rho_E^{(i)}}{\rho_0}\nabla^{WB}_{CW}p_0
			\right).
		\end{align}
	\end{subequations}
	Here $\rho_{\star}^{(i)}$ and $\bq_{\star}^{(i)}$ are updated from \eqref{S3_E8}.
	For this elliptic equation, all the derivative operators in \eqref{rho2} are discretized with high order central differences, e.g., the compact central difference for the Laplacian operator as in \cite{boscarino2019high}. After obtaining $\rho_{2,I}^{(i)}$, we can update $U_I^{(i)}$ from \eqref{S3_E9} explicitly in a sequential way. The detailed step-by-step flowchart of the fully discrete high order scheme is provided in the supplementary material.
	
	Taking the two dimensional model \eqref{F2d}-\eqref{s2d} described above as an example, these WENO operators $\nabla_{CW}$, $\nabla_{W}$, $\nabla_{UW}$ and $\nabla^{WB}_{CW}$ are defined as follows.
	\begin{itemize}
		\item {\bf{Characteristic-wise WENO reconstruction $\nabla_{CW}$. }} \\
		The traditional characteristic-wise WENO reconstruction is applied to those conservative variables of the explicit part. Let us denote
		$U_{EC} = (\rho, \rho u, \rho v, E)^T_E$ and
		\begin{equation}
			\label{FEC}
			\mathcal{F}_{EC}^x = (\rho u, \rho u^2 + p, \rho uv, (E+P)u)^T_E, \quad
			\mathcal{F}_{EC}^y = (\rho v, \rho uv, \rho v^2 + p, (E+P)v)^T_E,
		\end{equation}
		with
		\begin{equation}
			\label{SEC}
			\mathcal{S}_{EC}^x = (0,-\rho\partial_x\Phi,-\rho u\partial_x\Phi,0)^T_E, \quad \mathcal{S}_{EC}^y = (0,-\rho\partial_y\Phi,-\rho v\partial_y\Phi,0)^T_E.
		\end{equation}
		We note that $\mathcal{F}_{EC}^x$ and $\mathcal{F}_{EC}^y$ are the fluxes of 2D compressible Euler equations in a conservative form. They are exactly the first terms on the right-hand side of \eqref{FEx} and \eqref{FEy} (with a factor $\eps^2$), corresponding to the conserved variable $U_{EC}$. We take $\mathcal{F}_{EC}^x$ and $\mathcal{F}_{EC}^y$ to form our Jacobian matrices and perform a 2D characteristic-wise finite difference WENO reconstruction, as detailed in \cite{jiang1996efficient,shu1998essentially}, to obtain their numerical approximation. The same reconstructed values are also used to evaluate the second terms, e.g., $\partial_x(\rho u^2+p)$ and $ \partial_x(\rho u v)$ in the second term on the right side of \eqref{FEx} (with a factor $(1-\eps^2)$). Such a characteristic-wise WENO reconstruction is denoted by $\nabla_{CW}$. $\nabla_{CW}$ is very important for essentially non-oscillatory shock capturing, compared with component-wise WENO reconstructions, especially for multi-dimensional cases \cite{boscarino2022high}. 
		
		\item {\bf{Component-wise WENO reconstructions $\nabla_{W}$ and $\nabla_{UW}$. }} \\
		When the characteristic direction is unclear, the component-wise finite difference WENO reconstruction is applied to some first-order derivative terms, e.g., the third terms in \eqref{FEx} and \eqref{FEy} and those in \eqref{FIx} and \eqref{FIy}. We refer to \cite{jiang1996efficient,shu1998essentially} for detailed reconstruction procedures. Such reconstructions with a Lax-Friedrichs flux is frequently used. Denoting
		\begin{equation}
			\label{FIC}
			\mathcal{F}_{IC}^x = (\rho u, p_2, 0, (E+p)u)^T_I, \quad
			\mathcal{F}_{IC}^y = (\rho v, 0, p_2, (E+p)v)^T_I,
		\end{equation}
		and taking $\mathcal{F}_{IC}^{x}$ as an example, a Lax-Friedrichs flux splitting is given by
		\begin{equation}
			(\mathcal{F}_{IC}^{x})^{\pm}_{ij} = \frac{1}{2}\Big((\mathcal{F}_{IC}^x)_{ij} \pm \Lambda (U_{EC})_{ij}\Big),
		\end{equation}
		here we take $\Lambda = \max\limits_{\rho, \bu, p} \{|\bu|+\min(1,1/\eps)\sqrt{\frac{\gamma p}{\rho}} \}$ and use $U_{EC}$ to obtain numerical viscosities. For component-wise WENO reconstructions with a Lax-Friedrichs numerical flux, we denote them as $\nabla_{W}$. On the other hand, the component-wise WENO reconstruction with an upwind numerical flux is used in the equation of $\theta_2$, and we denote it as $\nabla_{UW}$, e.g. the $\nabla_{UW}\theta_{2,E}$ term in \eqref{hoh2}. 
		
		\item {\bf{Well-balanced WENO reconstruction. }} \\
		For all Mach Euler equations with static gravitational fields \eqref{FEe2}, it admits an equilibrium steady-state solution \eqref{equistate}. Namely, if the equilibrium \eqref{equistate} holds initially, it will hold for all later times. Numerically, well-balanced schemes are designed in order to preserve such a steady-state solution on the discrete level. Here we follow the approach proposed in \cite{xing2013high,li2018well} to construct well-balanced WENO approximation. The main idea is to use exactly the same reconstructions to the flux and source terms, and revise the numerical viscosity to a prebalanced form so that the numerical viscosities vanish for equilibrium solutions. For the system \eqref{FEe2}, an equilibrium static solution would have $\rho=\rho_0(\bx)$ and $p=p_0(\bx)$, which satisfies
		\begin{equation}
			\label{ES3}
			\nabla p_0 = -\rho_0\nabla\Phi.
		\end{equation}
		This is consistent with the asymptotic limit \eqref{ES2}, namely, the equilibrium solution is also a solution to \eqref{FEe_LE1} \cite{bispen2017asymptotic}. Here we still take $\partial_x\mathcal{F}^x_{EC}$ as an example to demonstrate how such well-balanced WENO scheme is designed. First, a well-balanced Lax-Friedrichs flux splitting is based on
		\begin{equation}
			\label{fixsplit}
			(\mathcal{F}_{EC}^{x})^{\pm}_{ij} = \frac{1}{2}\Big((\mathcal{F}_{EC}^x)_{ij} \pm \Lambda (\tilde{U}_{EC})_{ij}\Big),
		\end{equation}
		where $\tilde{U}_{EC} = (\rho - \rho_0, \rho\bu, E -\frac{p_0}{\gamma-1})^T_E$. For the source term \eqref{s2d}, we have
		$\nabla \Phi = -\nabla p_0/\rho_0$ from \eqref{ES3},
		and can rewrite \eqref{s2d} as
		\begin{equation}
			\mathcal{S}_{EC}^x = (0,-\frac{\rho}{\rho_0}\partial_x p_0,\rho u\partial_x \Phi,0)^T, \quad \mathcal{S}_{EC}^y = (0,-\frac{\rho}{\rho_0}\partial_y p_0,0,\rho v\partial_y \Phi)^T.
		\end{equation}
		To approximate $\mathcal{S}_{EC}^x$ numerically, in analogous to \eqref{fixsplit}, we split $p_0$ as 
		\begin{equation}
			(p_0)^\pm_{ij} =\frac{1}{2}(p_0)_{ij}.
		\end{equation}
		The finite difference characteristic-wise WENO reconstructions used to approximate $\{(\mathcal{F}_{EC}^{x})^{\pm}_{ij}\}$ will then be applied to $\{(p_0)^\pm_{ij}\}$ with exactly the same weights (evaluated based on the smoothness indicators of $\{(\mathcal{F}_{EC}^{x})^{\pm}_{ij}\}$). In this way, we can obtain the exact balance of these two terms on the discrete level, when the equilibrium state is reached. We refer to \cite{xing2013high,li2018well} for more detailed discussions. Similar approximations of $\partial_y\mathcal{F}^y_{EC}$ and $\mathcal{S}_{EC}^y$ can be applied. For those terms with above described well-balanced WENO reconstructions, we denote them as $\nabla_{CW}^{WB}$. 
		
	\end{itemize}
	

	\section{Asymptotic preserving and asymptotically accurate properties}
	\label{sec4}
	\setcounter{equation}{0}
	\setcounter{figure}{0}
	\setcounter{table}{0}
	
	In this section, we will prove that our first order IMEX scheme \eqref{IMEX1} is AP and the high order IMEX-RK scheme in Section \ref{sec3.3} is AA. 
	For any variable $w$, let us denote $w^n(\bx)=w(\bx,t^n)$. During the discussion, we focus on time discretization and keep the space continuous. 
	The initial solutions are assumed to be well-prepared, namely, the following expansions 
	\begin{subequations}
		\label{wp1}
		\begin{align}
			&\rho^n(\bx)=\rho^n_0(\bx)+\eps^2 \rho^n_2(\bx), \\
			&(\rho\bu)^n(\bx)=(\rho\bu)^n_0(\bx)+\eps (\rho\bu)^n_1(\bx), \\
			& p^n(\bx)=p^n_0(\bx)+\eps^2 p^n_2(\bx),
		\end{align}
	\end{subequations}
	hold at $n=0$. The relation $p=(\rho\theta)^\gamma$ yields
	\begin{equation}
		\label{wp2}
		\theta^n(\bx)=\theta^n_0(\bx)+\eps^2 \theta^n_2(\bx).
	\end{equation}
	In addition, the initial conditions should satisfy 
	\begin{equation}
		\label{RL}
		\nabla\cdot(\rho\bu)_0^0 = 0, \quad \nabla p^0_0 = -\rho_0^0 \nabla \Phi.
	\end{equation}
	\subsection{Asymptotic preserving}
	We start with proving the AP property of the first order IMEX scheme \eqref{IMEX1}. 
	\begin{thm}
		\label{thm41}
		With well-prepared initial conditions \eqref{wp1}-\eqref{wp2} at $n=0$ which satisfy \eqref{RL}, and assuming the expansion \eqref{wp1}-\eqref{wp2} at all later times, the first order IMEX scheme \eqref{IMEX1} with \eqref{eqr2p2}, under the assumptions of \eqref{pr}, $\nabla p_0=-\rho_0\nabla \Phi$ and $\theta_0$ is a constant, is asymptotic preserving, namely, as $\eps\rightarrow 0$, the limiting scheme of \eqref{IMEX1} is a consistent discretization of \eqref{FEe_LE1}. 
	\end{thm}
	
	\begin{proof}
		As $\eps\rightarrow 0$, under the assumptions that $\nabla p_0=-\rho_0\nabla \Phi$, $\theta_0$ is a constant, the scheme \eqref{IMEX1} becomes 
		\begin{subequations}
			\label{IMEX-RK1_LM}
			\begin{align}
				\label{S4_4}
				&\frac{\rho_0^{n+1} - \rho_0^n}{\Dt} + \nabla\cdot(\rho\bu)_0^{n+1} = 0,
				\\
				\label{S4_5}
				&\frac{(\rho\bu)_0^{n+1} - (\rho \bu)_0^n}{\Dt}+\nabla \cdot(\rho \bu \otimes \bu+p\mI)_0^n+\nabla p^{n+1}_2 =  - \rho^{n}_0 \nabla\Phi- \rho^{n+1}_2 \nabla\Phi,
				\\
				\label{S4_6}
				&\frac{E_0^{n+1} - E_0^n}{\Dt}
				+\nabla\cdot \left((E_0+p_0)\bu_0\right)^{n+1}  = -\left(\rho\bu\right)_0^{n+1}\cdot \nabla\Phi,
				\\
				\label{S4_8}
				&\frac{\theta_2^{n+1} - \theta_2^n}{\Dt}+ \bu_0^n\cdot\nabla\theta_2^n =0, 
			\end{align}
		\end{subequations}
		and the EOS \eqref{EOS2} reduces to (at the time levels $t^n$ and $t^{n+1}$)
		\begin{equation}
			\label{S4_7}
			E_0^{n} = \frac{p_0^{n}}{\gamma-1},  \qquad E_0^{n+1} = \frac{p_0^{n+1}}{\gamma-1}.
		\end{equation}
		
		We now prove that starting from the well-prepared initial conditions \eqref{wp1}-\eqref{RL} at $n=0$, the limiting scheme \eqref{IMEX-RK1_LM} is a consistent discretization to \eqref{FEe_LE1} after one time step. First of all, due to $\nabla p_0^0=-\rho_0^0\nabla \Phi$, \eqref{S4_5} becomes
		\begin{equation}
			\label{rhou0}
			\frac{(\rho\bu)_0^{1} - (\rho \bu)_0^0}{\Dt}+\nabla \cdot(\rho \bu \otimes \bu)_0^0+\nabla p^{1}_2 = - \rho^{1}_2 \nabla\Phi.
		\end{equation}
		Taking divergence on both sides of \eqref{rhou0} and using $\nabla\cdot(\rho\bu)^0_0=0$, we have
		\begin{equation}
			\label{rhou}
			\frac{\nabla\cdot(\rho\bu)_0^{1}}{\Dt}+\nabla\cdot(\nabla \cdot(\rho \bu \otimes \bu)_0^0)+\nabla\cdot(\nabla p^{1}_2 + \rho^{1}_2 \nabla\Phi)=0.
		\end{equation}
		As $\eps\rightarrow 0$, the elliptic equation \eqref{eqr2p2} becomes
		\begin{equation}
			\label{eqr2p2_2}
			0 = \Dt^2\,\nabla \cdot(\nabla\cdot (\rho\bu\otimes\bu)^0_0) + \Dt^2\nabla\cdot\left(\nabla p_2^{1} +  \rho_2^{1}\nabla\Phi\right),
		\end{equation}
		since $\widetilde{\rho} - (1-\eps^2)\,\Dt\,\nabla \cdot (\widetilde{\rho\bu})-\rho^0_0$ in \eqref{eqr2p2} approaches $\Dt^2\,\nabla \cdot(\nabla\cdot (\rho\bu\otimes\bu)^0_0)$, due to \eqref{Ustara} and \eqref{Ustarb}. 
		Substituting \eqref{eqr2p2_2} into \eqref{rhou} leads to
		\begin{equation}
			\label{rhou1}
			\nabla \cdot (\rho\bu)_0^{1} = 0,
		\end{equation}
		and from \eqref{S4_4}, we obtain $\rho^1_0=\rho^0_0$. 
		
		Utilizing \eqref{S4_7}, the energy equation \eqref{S4_6} becomes
		\begin{equation}
			\label{p0}
			\frac{1}{\gamma-1}\frac{p_0^{1} - p_0^0}{\Dt}
			+\nabla\cdot \left(\frac{\gamma\,p^1_0}{(\gamma-1) \rho^1_0}(\rho \bu)^1_0\right)  = -\left(\rho\bu\right)_0^{1}\cdot \nabla\Phi.
		\end{equation}
		If we assume a constant background potential temperature $\theta_0$ as in Section 2, from $p=(\rho\theta)^\gamma$, we have $p_0=C\,\rho^\gamma_0$ with $C=\theta_0^\gamma$, which leads to
		\begin{equation}
			\nabla\left(\frac{p^1_0}{\rho_0^1}\right)=C\nabla (\rho^1_0)^{(\gamma-1)}=C(\gamma-1)(\rho^1_0)^{\gamma-2}\nabla\rho_0^1,
		\end{equation}
		and using \eqref{rhou1}, we have
		\begin{align}
			\nabla\cdot \left(\frac{\gamma\,p^1_0}{(\gamma-1) \rho^1_0}(\rho \bu)^1_0\right)&=\frac{\gamma}{\gamma-1}\left((\rho \bu)^1_0\cdot\nabla\left(\frac{ p^1_0}{\rho^1_0}\right)+\frac{ p^1_0}{\rho^1_0}\nabla\cdot(\rho \bu)^1_0\right) \\
			&=C\gamma(\rho^1_0)^{\gamma-2}(\rho \bu)^1_0\cdot\nabla\rho_0^1. \notag
		\end{align}
		And $\nabla p^1_0=-\rho^1_0\nabla\Phi$ yields
		\begin{equation}
			-\left(\rho\bu\right)_0^{1}\cdot \nabla\Phi=(\rho \bu)^1_0\cdot\frac{\nabla p^1_0}{\rho^1_0}=C\gamma(\rho^1_0)^{\gamma-2}(\rho \bu)^1_0\cdot\nabla\rho_0^1,
		\end{equation}
		therefore, the equation \eqref{p0} leads to $p^1_0=p^0_0$. Note that the update of density equation in \eqref{S4_4} yields $\rho^1_0=\rho^0_0$, and the update of energy equation in \eqref{S4_6} yields $p^1_0=p^0_0$, which is consistent with the assumption $\nabla p_0^1 = -\rho^1_0\nabla \Phi$ at the next time level.
		
		From \eqref{rhou1}, \eqref{rhou0} and \eqref{S4_8}, we observe that the limiting scheme \eqref{IMEX-RK1_LM} is a consistent discretization to \eqref{FEe_LE1} at the next time level, as well as $p^1_0=p^0_0$ and $\rho^1_0=\rho^0_0$. By the mathematical induction, we can prove that the same conclusion holds for all later time levels in a similar way, under the assumptions of the expansion \eqref{wp1}-\eqref{wp2} at any later time level $t^n$, which confirms the AP property of our scheme.
	\end{proof}
	
	\subsection{Asymptotically accurate}
	The AP property can ensure the first-order scheme \eqref{IMEX1} to be a consistent discretization of the limiting equation \eqref{FEe_LE1} when $\eps\rightarrow 0$. The high order scheme \eqref{S3_E7}-\eqref{S3_E15} is named AA \cite{pareschi2005implicit,boscarino2019high,boscarino2022high}, if it is AP and preserves the high order temporal accuracy as $\eps$ goes to zero.
	
	
	\begin{thm}
		\label{thm42}
		Consider the high order semi-implicit IMEX-RK scheme \eqref{S3_E7}-\eqref{S3_E15} with space continuous, and use a $k$-th order stiffly accurate IMEX-RK scheme \eqref{DBT}. Under the same assumptions for the AP property, for the solutions after one time step, we have
		\begin{equation*}
			\lim_{\eps \to 0} \theta^1(\bx;\eps)=\theta_0, \, \lim_{\eps \to 0} \nabla \cdot \Big(\rho^1(\bx;\eps)\bu^1({\bx};\eps)\Big)=0,
			\,
			\lim_{\eps \to 0} (\nabla p^1(\bx;\eps)+\rho^1(\bx;\eps)\nabla\Phi)=0.
		\end{equation*}
		If we denote 
		${\bf V}^1({\bf x};\eps)=\big(\rho^1(\bx;\eps),\rho^1(\bx;\eps)\bu^1(\bx;\eps),p^1(\bx;\eps), \theta_2^1(\bx;\eps)\big)
		$
		and let 
		$\bV^{e}(\bx,t)=\big(\rho^{e}(\bx,t),\rho^{e}(\bx,t)\bu^{e}(\bx,t),p^{e}(\bx,t),\theta_2^{e}(\bx,t)\big)$
		be the exact solutions of \eqref{FEe_LE1}, one has the one-step error estimate
		\begin{equation}
			\label{Prop2}
			\lim_{\eps\to 0} {\bf V}^1({\bf x};\eps)={\bf V}^{e}({\bf x}, \Delta t)+\mathcal{O}(\Delta t^{k + 1}),
		\end{equation}
		namely, the high order semi-implicit scheme is AA.
	\end{thm}
	
	The proof follows from the mathematical induction and similar derivations of the AP analysis. A detailed proof is provided in the supplementary material. 

	\section{Numerical tests}
	\label{sec5}
	\setcounter{equation}{0}
	\setcounter{figure}{0}
	\setcounter{table}{0}
	
	In this section, we will perform some numerical tests for the all-Mach full Euler equations with gravity, where the Mach number ranges from $0$ to $\mO(1)$. For one-dimensional (1D) problems, we consider the vertical direction with coordinate denoted by $x$. For two-dimensional (2D) problems, we use the coordinate $(x,y)$ with $y$ being the vertical direction, unless specified otherwise.
	
	The proposed method was outlined in Section \ref{sec3.4}. The third order SA IMEX-RK scheme in \cite[Section 3.2.2]{boscarino2022high} will be used as temporal discretization, 
	and the fifth order finite difference WENO scheme \cite{shu1998essentially} is adopted in space. The time step is taken to be $\Delta t = \text{CFL}\,\Delta x/\Lambda$ with  $\Lambda = \max \{|\bu|+\min(1,1/\eps)\sqrt{{\gamma p}/{\rho}}\}$,
	and we set $\text{CFL}=0.2$. Our proposed scheme is denoted as the ``IMEX'' scheme in the following. We will compare our scheme to an explicit fifth order well-balanced finite difference WENO scheme, as was developed in \cite{li2018well}, which is denoted as ``WB-Xing''. We will demonstrate that our approach is high order accurate, AP and AA, as well as that it is more computationally efficient as compared to WB-Xing in the low Mach regime.
	
	\subsection{1D accuracy test}
	\label{exam1}
	We first consider the following 1D example with a linear gravitational field $\Phi_x=1$. The hydrostatic steady state satisfying \eqref{ES2} is given by 
	\begin{equation}
		\label{steady_state1}
		\rho_0(x) = \left(1-\frac{\gamma -1}{\gamma}x\right)^{\frac{1}{\gamma-1}}, \quad
		p_0(x)    = \left(1-\frac{\gamma -1}{\gamma}x\right)^{\frac{\gamma}{\gamma-1}}, \quad
		\theta_0(x) = 1,
	\end{equation}
	with the adiabatic index $\gamma=1.4$. The corresponding perturbations are also at a steady state, where
	\begin{equation*}
		\rho_2(x) = 1 + 0.2\sin(\pi x ),\qquad
		p_2(x) = 4.5 - x + 0.2\cos(\pi x)/\pi,
	\end{equation*}
	so that we have steady-state well-prepared solutions for \eqref{S3_1}
	\begin{equation*}
		\rho(x) = \rho_0(x) + \eps^2\rho_2(x),
		u(x)=0, 
		p(x) = p_0(x) + \eps^2p_2(x),
		\theta_2(x) = \frac{1}{\rho}\left(\frac{\rho_0\theta_0}{\gamma p_0}p_2 - \rho_2\theta_0\right).
	\end{equation*}
	Here the computational domain is $\Omega=[0,2]$. We divide $\Omega$ into $N$ uniform cells, where $N=8\times2^i \,(i=1,2,\ldots,5)$. The errors are computed by comparing the numerical solution to the exact solution at a final time $T=0.1$. We show the $L_1$ errors and orders for the density $\rho$ in Table~\ref{T_eg1_1}, by taking four different global Mach numbers $\eps=1,  10^{-2}, 10^{-4}$ and $0$. We observe that our scheme has at least fifth order accuracy for all $\eps$'s, as the spatial error is dominant and the solutions reach the steady state. The convergence orders of $\rho u$ and $E$ are similar, and we omit them to save space.
	
	\renewcommand{\multirowsetup}{\centering}
	\begin{table}[htbp]
		\caption{ Example \ref{exam1}. $L_1$ errors and orders of $\rho$  with $\eps=1, 10^{-2}, 10^{-4}, 0$. $T=0.1$.}
		\begin{center}
			\begin{tabular}{c|c|c|c|c|c|c|c|c}
				\hline
				\multicolumn{1}{c|}{\multirow{2}*{N}}&\multicolumn{2}{c|}{ $\eps=1$}&\multicolumn{2}{c|}{$\eps=10^{-2}$}&\multicolumn{2}{c|}{$\eps=10^{-4}$} &\multicolumn{2}{c}{$\eps=0$}\\
				\cline{2-9}
				\multicolumn{1}{c|}{} &error& order& error&order&error& order&error& order \\  \cline{1-9}
				16 &     9.19E-05 &     --&     1.20E-08 &     --&     1.33E-08 &     --&      1.33E-08 &     --\\
				32 &     2.94E-06 &     4.96&     2.94E-10 &     5.35&     9.63E-11 &     7.11&     9.65E-11 &     7.11 \\
				64 &     8.90E-08 &     5.05&     5.80E-12 &     5.66&     4.25E-12 &     4.50&     4.25E-12 &     4.51  \\
				128 &     2.74E-09 &     5.02&     8.99E-14 &     6.01&     6.03E-14 &     6.14&     6.04E-14 &     6.14   \\
				256 &     8.52E-11 &     5.01&     1.64E-15 &     5.77&     7.30E-16 &     6.37&     7.33E-16 &     6.36  \\
				\hline
			\end{tabular}
		\end{center}
		\label{T_eg1_1}
	\end{table}

	\subsection{1D shock tube problem}
	\label{1D_shock}
	In this example, we consider a 1D shock tube problem in a high Mach regime where we take $\eps=0.9$, and the initial conditions are given by
	\begin{equation*}
		(\rho,u,p) = (1,0,1),\quad \text{ if } x<0.5; \qquad (0.125,0,0.1), \quad \text{ otherwise; }
	\end{equation*}
	which has been studied in \cite{boscarino2022high} without the gravitational source. Here we consider a linear gravity with $\Phi_x =1$. The background steady state is still \eqref{steady_state1}. The computational domain is $[0,1]$, with inflow and outflow boundary conditions on the left and right boundaries respectively. We take $N=200$ and show the results at $T=0.1$ in Fig.~{\ref{Fig_1D_shock}}. We can observe that the IMEX numerical solutions match those of WB-Xing very well, and both schemes work well in the high Mach regime when the shock exists.
	
	\begin{figure}[hbtp]
		\begin{center}
			{\includegraphics[width=5.5cm]{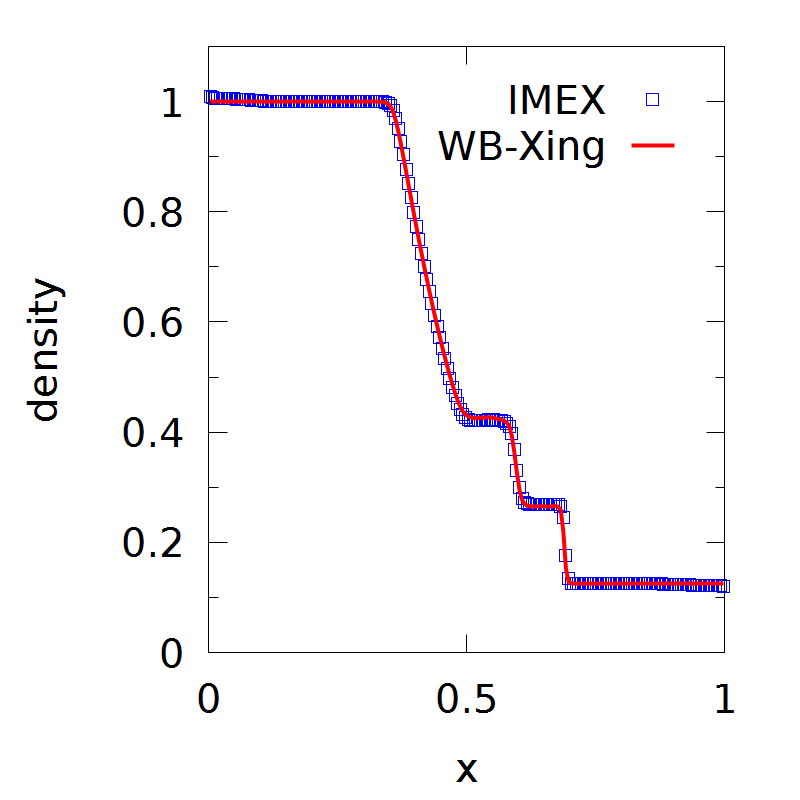}}
			{\includegraphics[width=5.5cm]{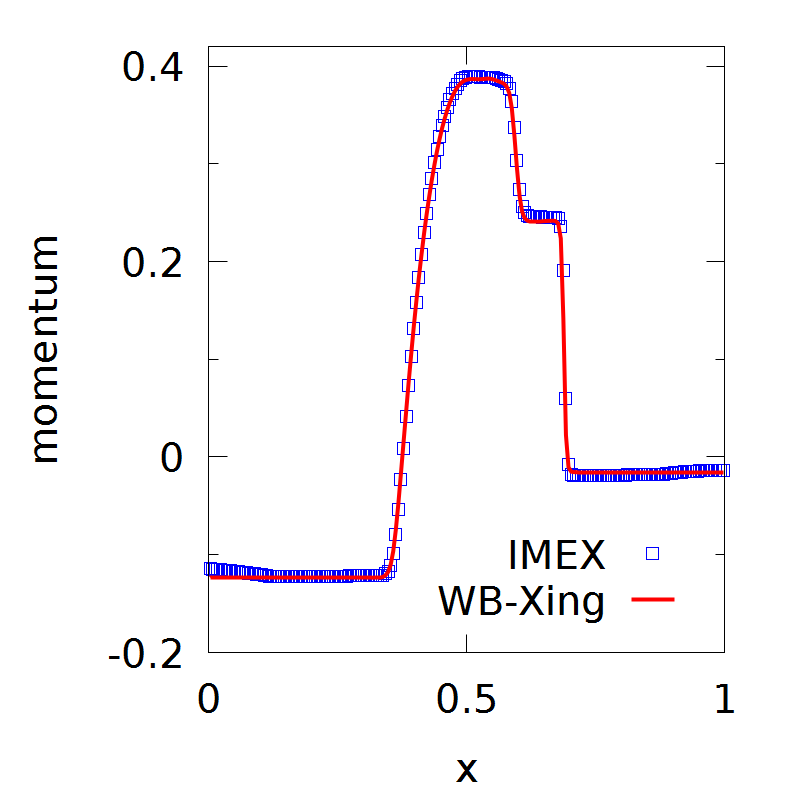}}
			{\includegraphics[width=5.5cm]{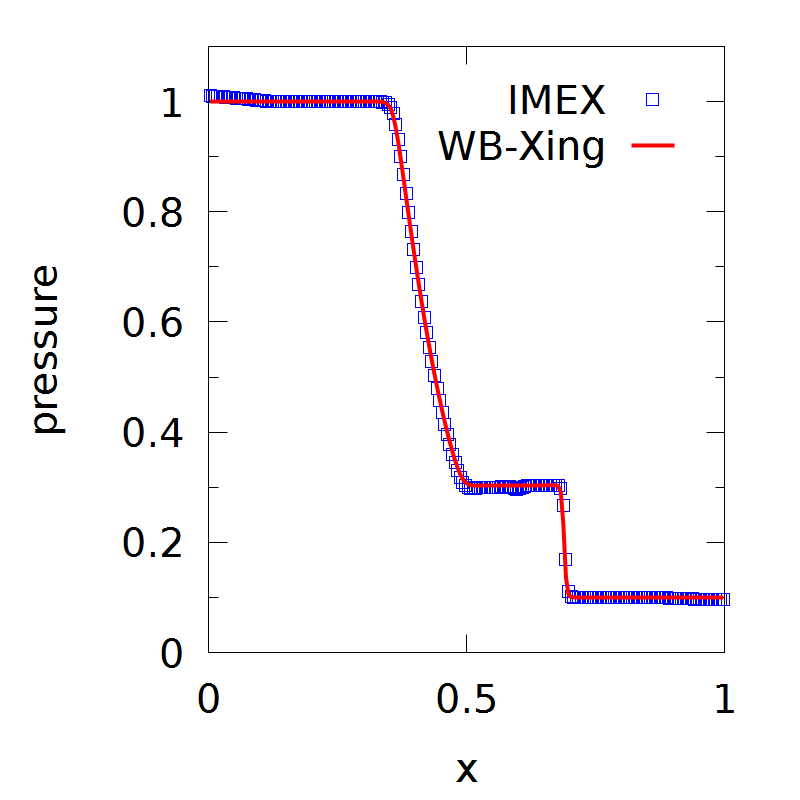}}
			\caption{ Example~\ref{1D_shock}. Numerical solutions of $\rho$, $\rho u$ and $p$ with $N=200$ at $T=0.1$. }
			\label{Fig_1D_shock}
		\end{center}
	\end{figure}
	
	\subsection{2D accuracy test}
	\label{exam3}
	In this example, we consider a smooth 2D problem with a linear gravitational field $(\Phi_x,\Phi_y) = (1,1)$. The steady state is given by
	\begin{equation*}
		\rho_0(x,y) = \left(1-\frac{\gamma -1}{\gamma}(x+y-2)\right)^{\frac{1}{\gamma-1}}, \,\,
		p_0(x,y)    = \left(1-\frac{\gamma -1}{\gamma}(x+y-2)\right)^{\frac{\gamma}{\gamma-1}}, 
	\end{equation*}
	with $\theta_0(x,y)=1$ and the adiabatic index $\gamma=1.4$. Similar to Example \ref{exam1}, the perturbations are also at a steady state
	\begin{equation*}
		\rho_2(x,y) = 1 + 0.2\sin(\pi (x + y) ),\quad
		p_2(x,y) = 4.5 - x + 0.2\cos(\pi (x +y))/\pi,
	\end{equation*}
	so that the steady-state exact solutions are
	\begin{equation*}
		\rho(x,y) = \rho_0(x,y) + \eps^2\rho_2(x,y),
		p(x,y) = p_0(x,y) + \eps^2p_2(x,y),
		\theta_2 =\frac{1}{\rho}\left(\frac{\rho_0\theta_0}{\gamma p_0}p_2 - \rho_2\theta_0\right).
	\end{equation*}
	Here $\bu= (u_0,-u_0)$ which is perpendicular to the gravitational field, and we take $u_0=1$. The computational domain is $\Omega  = [0, 2]^2$, which is divided into $N^2$ uniform cells, with $N=8\times2^i\,(i=1,2,\ldots,5)$. The errors are computed by comparing the numerical solution to the exact solution at the final time $T=0.5$. Similarly we show the $L_1$ errors and orders of $\rho$ in Tables~\ref{T_eg3_1}, with four different global Mach numbers $\eps=1,  10^{-2}, 10^{-4}$ and $0$. We can observe the expected convergence rates for all $\eps$'s, similar to the 1D case.
	
	\renewcommand{\multirowsetup}{\centering}
	\begin{table}[htbp]
		\caption{ Example \ref{exam3}. $L_1$ errors and orders of $\rho$ with $\eps=1, 10^{-2}, 10^{-4}, 0$. $T=0.05$. $N^2$ uniform cells.}
		\begin{center}
			\begin{tabular}{c|c|c|c|c|c|c|c|c}
				\hline	
				\multicolumn{1}{c|}{\multirow{2}*{$N$}}&\multicolumn{2}{c|}{ $\eps=1$}&\multicolumn{2}{c|}{$\eps=10^{-2}$}&\multicolumn{2}{c|}{$\eps=10^{-4}$} &\multicolumn{2}{c}{$\eps=0$}\\
				\cline{2-9}
				\multicolumn{1}{c|}{} &error& order& error&order&error& order&error& order \\  \cline{1-9}
				$16$ &     4.86E-04 &       --&     1.83E-06 &       --&     1.83E-06 &       --&     1.83E-06 &       --\\
				$32$ &     3.70E-05 &     3.72&     8.17E-08 &     4.49&     8.14E-08 &     4.49&     8.14E-08 &     4.49\\
				$64$ &     1.48E-06 &     4.65&     2.32E-09 &     5.14&     2.31E-09 &     5.14&     2.31E-09 &     5.14 \\
				$128$&     5.51E-08 &     4.74&     5.48E-11 &     5.41&     5.41E-11 &     5.42&     5.41E-11 &     5.42  \\
				$256$&     1.38E-09 &     5.32&     1.09E-12 &     5.66&     1.06E-12 &     5.67&     1.06E-12 &     5.67
				\\\hline
				
			\end{tabular}
		\end{center}
		\label{T_eg3_1}
	\end{table}

	\subsection{Traveling vortex}
	\label{2D_vortex} This test has been considered for the 2D shallow water equations \cite{huang2022high}. 
	In \eqref{S3_1}, if we assume $\theta=1$ and take $c_p = 2, c_v=1$, then the specific gas constant is $R=c_p-c_v=1$, and $\gamma=c_p/c_v=2$, it is equivalent to the shallow water equations and in this case the pressure is $p=\rho^2/2$ with $\eps=0.05$. Following the settings in \cite{huang2022high}, for the source term we take 
	\begin{equation*}
		\Phi(x,y) = e^{-5(x-1)^2},
	\end{equation*}
	namely, the gravity is along the $x$-direction for this example. 
	Other initial conditions are given by
	\begin{align*}
		&\rho(x,y,0) = 110  - \Phi +
		\left\{
		\begin{aligned}
			&\left(\frac{\eps \Gamma}{\omega}\right)^2(k(\omega\Gamma_c) - k(\pi)), \quad & \text{if }  \omega\Gamma_c \le\pi;\\
			&0, \quad & \text{otherwise},
		\end{aligned}
		\right.
		\\
		&u(x,y,0) = 2+
		\left\{
		\begin{aligned}
			&\Gamma(1+\cos(\omega\Gamma_c))(0.5-y), \quad & \text{if }  \omega\Gamma_c \le\pi;\\
			&0, \quad & \text{otherwise},
		\end{aligned}
		\right.
		\\
		&v(x,y,0) =
		\left\{
		\begin{aligned}
			&\Gamma(1+\cos(\omega\Gamma_c))(x-0.5), \quad & \text{if }  \omega\Gamma_c \le\pi;\\
			&0, \quad & \text{otherwise},
		\end{aligned}
		\right.
	\end{align*}	
with
$
\Gamma_c = \sqrt{(x-0.5)^2+(y-0.5)^2},\,  \Gamma = 8,\, \omega = 4\pi,
$
and
\[
k(\xi) = 2\cos(\xi) + 2\xi\sin(\xi) + \frac{1}{8}\cos(2\xi) + \frac{\xi}{4}\sin(2\xi)+\frac{3}{4}\xi ^2.
\]
The hydrostatic steady state takes the form
\begin{equation*}
	\rho_0(x,y) = 110 - \Phi,\qquad \theta_0 =1, \qquad
	p_0(x,y) = \frac{1}{2}\left(110 - \Phi\right)^2.
\end{equation*}
Here the computational domain is $[0,2]\times [0,1]$, with periodic boundary conditions. We use this example to demonstrate that the new modeling \eqref{S3_1} and our proposed scheme can have similar good performances as in \cite{huang2022high} for this special case.
We strictly follow our scheme in Section \ref{sec3.4} with a mesh grid $200\times100$ and show the perturbation of density $\rho-\rho_0$ in Fig.~{\ref{Fig2D_vortex_2}}. We can observe that the numerical results are very similar to those obtained in \cite{huang2022high}, which demonstrates the nice performance of the proposed AP method in the low Mach regime. 
\begin{figure}[hbtp]
	\begin{center}
		{\includegraphics[width=5.5cm]{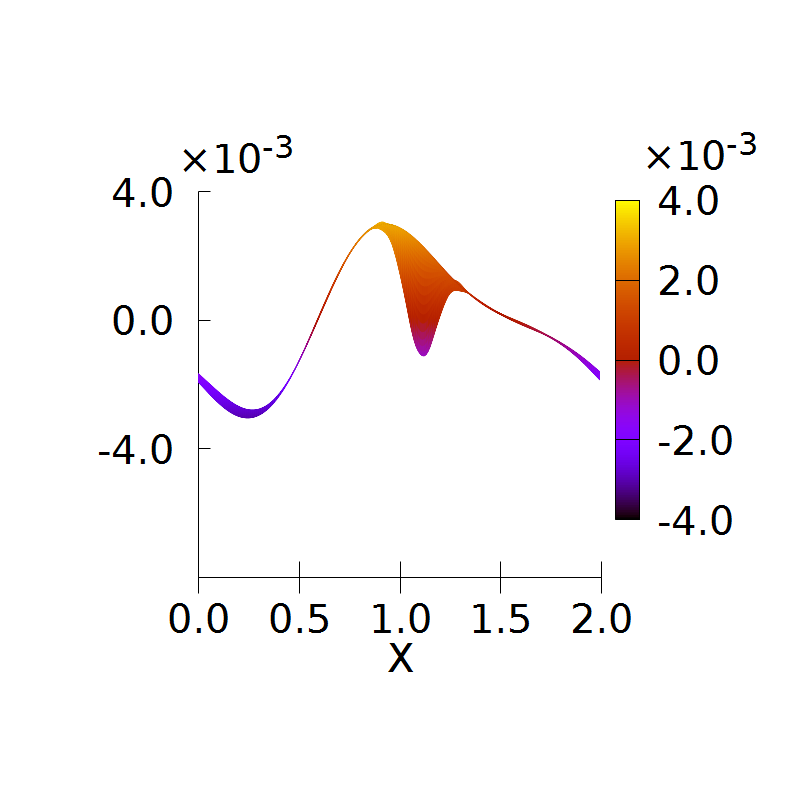}}
		{\includegraphics[width=5.5cm]{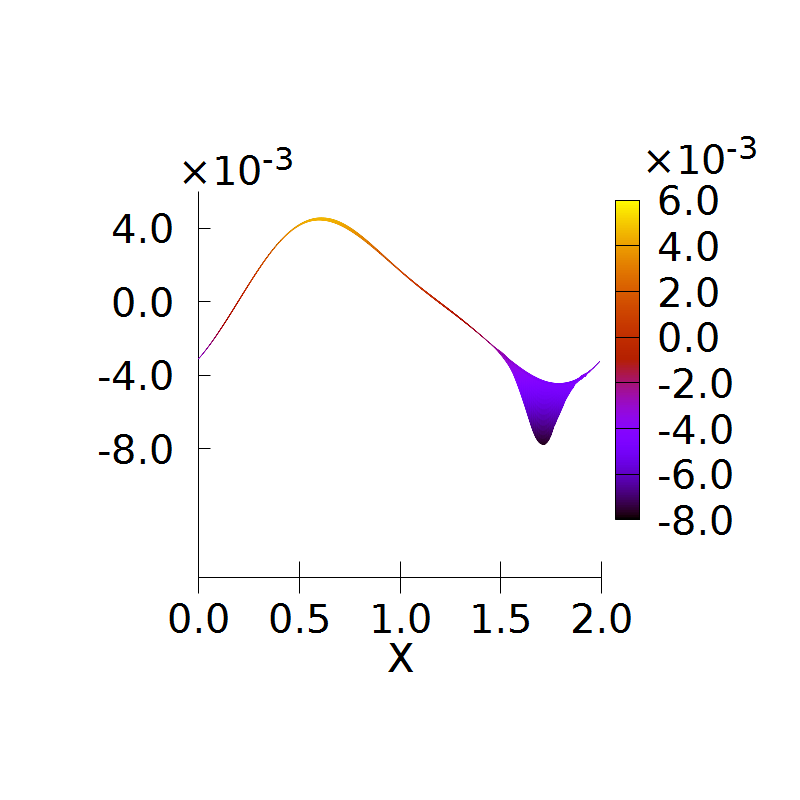}}
		{\includegraphics[width=5.5cm]{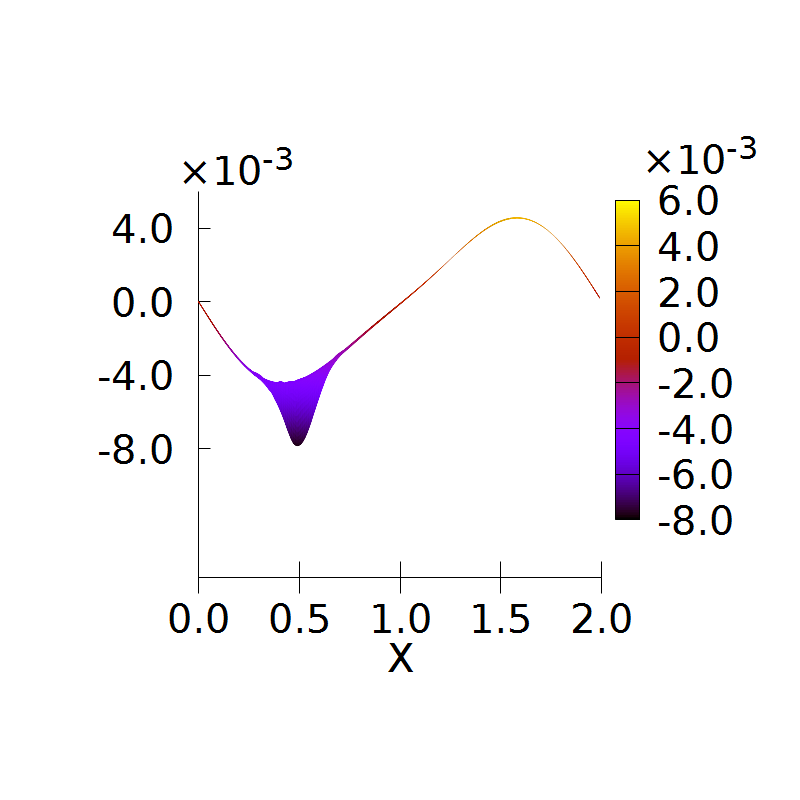}}
		\caption{ Example~\ref{2D_vortex}. Numerical solutions for the perturbation of density $\rho-\rho_0$. Mesh $200\times 100$. From left to right: $t= 0.3,  0.6, 1.0$, respectively. }
		\label{Fig2D_vortex_2}
	\end{center}
\end{figure}

\subsection{2D isothermal equilibrium}
\label{exam5}
This example is designed to test the well-balanced property of a given scheme, as well as its ability to capture small disturbances around an equilibrium state \cite{wu2021uniformly}. We consider an isothermal equilibrium state 
\begin{equation}
	\label{exam5_ini}
	\rho(x,y)  = \bar{\rho}\exp\left(-\frac{\bar{\rho}g}{\bar{p}}(x+y)\right), \,
	\bu = 0,\,
	p(x,y) = \bar{p}\exp\left(-\frac{\bar{\rho}g}{\bar{p}}(x+y)\right),
\end{equation}
with $\bar{\rho} = 1.21$, $\bar{p} = 1$ and $g = 1$. The linear gravitational field with $(\Phi_x, \Phi_y) = (1, 1)$ is taken. The adiabatic index is $\gamma=1.4$ and the global Mach number is set to be $\eps=0.9$. The problem is defined on a unit square $[0,1]^2$. We run the simulation up to the final time $T=1$, on three different meshes $50\times50$, $100\times100$ and $200\times200$. We compare the numerical solutions to the steady-state exact solutions. It was observed that both IMEX and WB-Xing can preserve the solutions up to machine errors, and a well-balanced property is achieved for both schemes. We omit the results here to save space. 		

Next, we add a small perturbation to the pressure 
\begin{equation*}
	p_2(x,y) = \frac{1}{810}\exp\left(-\frac{100\bar{\rho}g}{\bar{p}}((x-0.3)^2+(y-0.3)^2\right)
\end{equation*}
which is a small Gaussian hump centered at $(0.3,0.3)$. 
The other terms are kept to be the same. We take a mesh grid $100\times100$ and compute the numerical solution up to $T=0.15$. Here a transmissive boundary condition \cite{toro2009riemann} is used. The numerical results are shown in Fig.~\ref{Fig5_1}, and again we can observe that both schemes can capture the small perturbations very well in this high Mach regime.


\begin{figure}[hbtp]
	\begin{center}
		{\includegraphics[width=5.5cm]{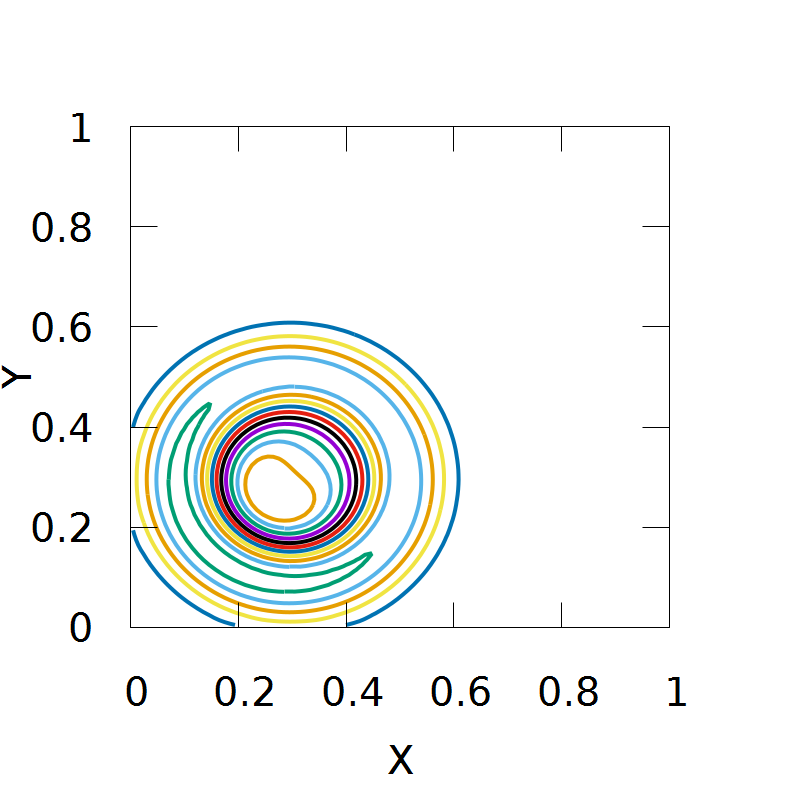}}
		{\includegraphics[width=5.5cm]{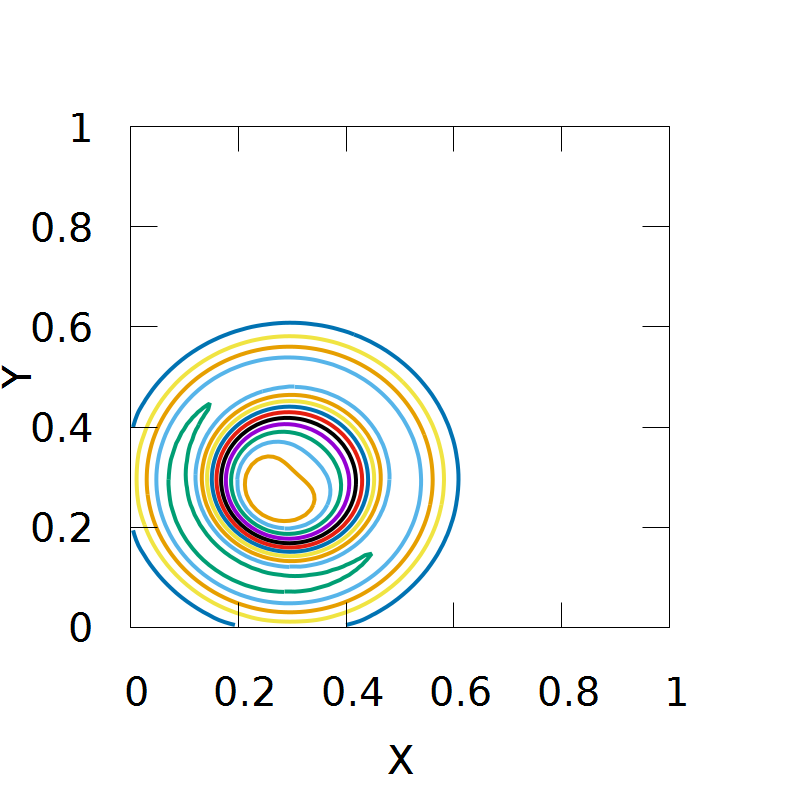}}
		{\includegraphics[width=5.5cm]{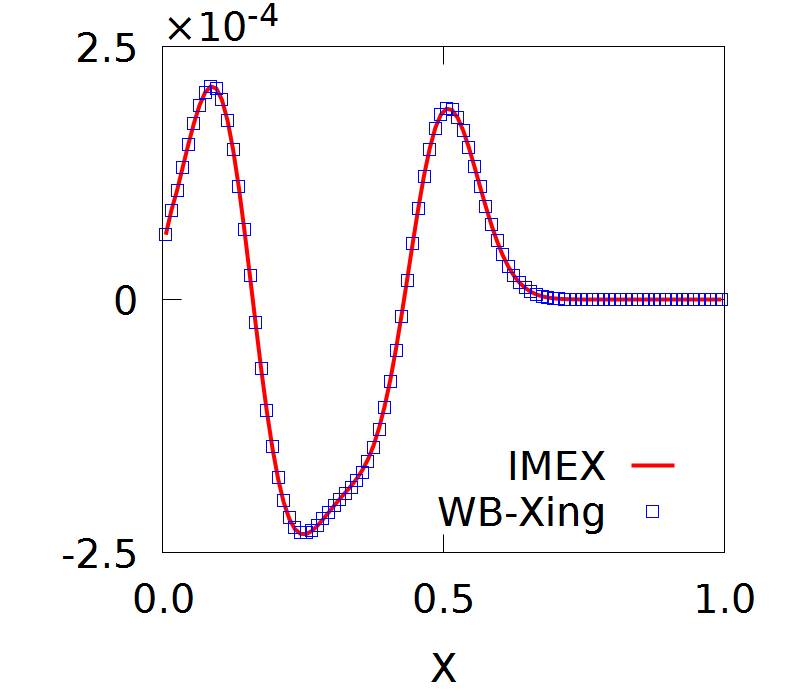}}
		{\includegraphics[width=5.5cm]{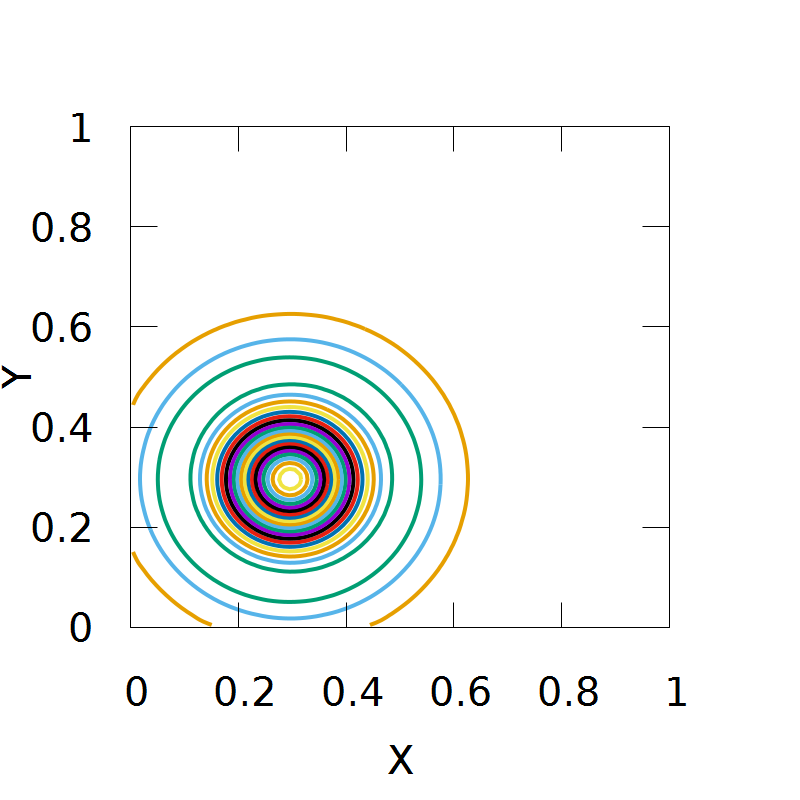}}
		{\includegraphics[width=5.5cm]{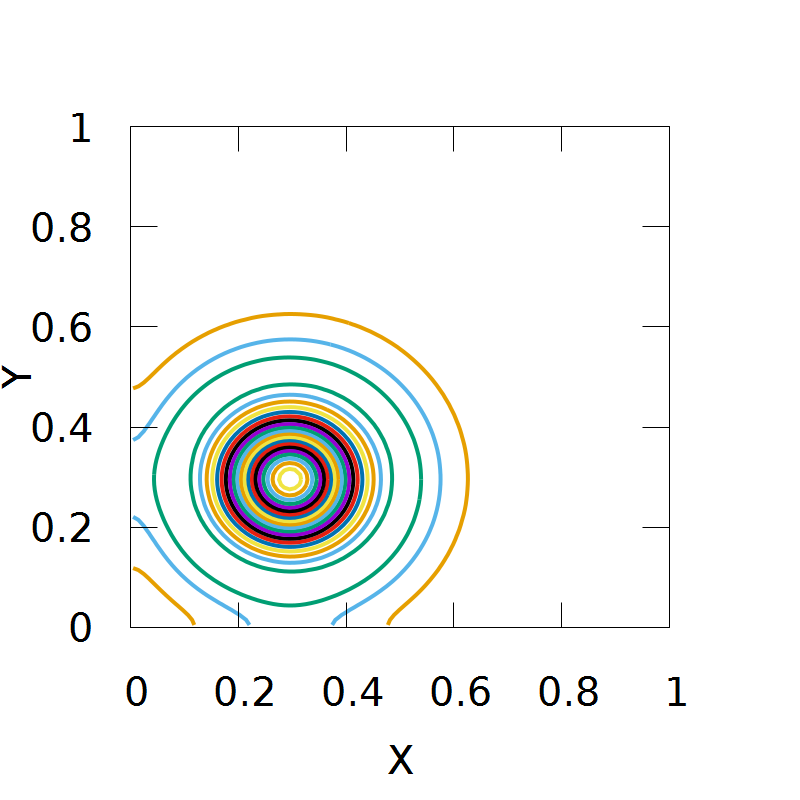}}
		{\includegraphics[width=5.5cm]{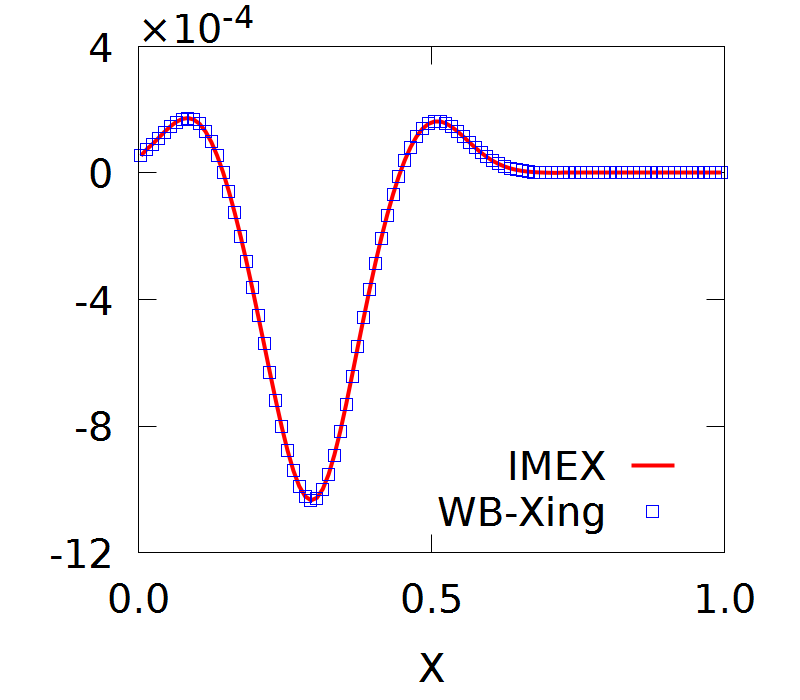}}
		\caption{ Example~\ref{exam5}. Numerical solutions of the pressure perturbation  and the density perturbation at $T=0.15$ with $N\times N=100\times 100$. Left: WB-Xing; middle: IMEX; right: cuts along the line $y=0.305 \, m$. Top: pressure perturbation; bottom: density perturbation.   }
		\label{Fig5_1}
	\end{center}
\end{figure}

\begin{rem}
	{
		For this 2D isothermal equilibrium example, numerical boundary conditions are treated as follows. We split our main variable $U=(\rho,\rho\bu,E,\theta)^T$ into two parts: a hydrostatic one $U^0$ and a perturbation $U'$, that is, $U=U^0+U'$ with
		\begin{equation*}
			U^0 = \left(
			\rho_0,
			{\bf 0},
			\frac{p_0}{\gamma-1},
			\theta_0
			\right)^T,\qquad
			U' = \left(
			\eps^2\rho_2,
			\rho\bu,
			E-\frac{p_0}{\gamma-1},
			\eps^2\theta_2
			\right)^T.
		\end{equation*}
		We apply a transmissive boundary condition \cite{toro2009riemann} to the perturbation $U'$, i.e., ghost values of $U'$ are assigned by mirror symmetry. Ghost values of the hydrostatic component $U^0$ are obtained via extrapolation. The same boundary treatments are also applied to the following two examples.
	}
\end{rem}

\subsection{Rising thermal bubble}
\label{exam6}
This is a benchmark test problem for atmospheric flows, simulating the dynamics of a warm bubble, which has been studied in \cite{ghosh2016well,giraldo2008study,wu2021uniformly}.  Here we consider a linear gravitational field in the vertical direction with $\Phi_x=0$ and $\Phi_y = g$, and the gravitational constant is $g=9.8m/s^2$.
The Exner pressure takes the form
$\Pi =1-(\gamma-1)gy/\gamma/R/\bar{T}$,
with $\bar{T} = 300 \, K$, $ \gamma =1.4,\,$ $R = 287.058 \,\text{J/kg K}$ being the gas constant, and the potential temperature is
\begin{equation*}
	\theta(x,y,0) = \theta_0 + \Delta\theta(x,y,0) = 300\,K +
	\left\{
	\begin{aligned}
		&0,   & r   > r_c;\\
		&\frac{\theta_c}{2}(1+cos(\pi\, r/r_c))   & r  \le r_c.
	\end{aligned}
	\right.
\end{equation*}
Here $r = \sqrt{(x-x_c)^2 + (y-y_c)^2}$, $\theta_c = 0.5 \,K$,  $(x_c,y_c)=(500,350)\,m$ and $r_c = 250\,m$, so that the initial condition can be denoted as
\begin{equation*}
	\rho(x,y,0) = \frac{\bar{p}}{R \theta}\, \Pi^{\frac{1}{\gamma-1}},\quad
	\bu(x,y,0) = (0,0),\quad
	p(x,y,0) = \bar{p} \,\Pi^{\frac{\gamma}{\gamma-1}},
\end{equation*}
with the hydrostatic steady state
\begin{equation*}
	\rho_0(x,y) = \frac{\bar{p}}{R\,\theta_0}\, \Pi^{\frac{1}{\gamma-1}},\qquad
	\theta_0(x,y) = 300 \,K,
	\qquad
	p_0(x,y) = \bar{p}\, \Pi^{\frac{\gamma}{\gamma-1}}.
\end{equation*}
where $\bar{p} = 10^5 N/m^2 $ is a reference pressure at $y=0\,m$, and the computational domain is $[0,1000]\times[0,1000]\, m^2$.

This problem can be rewritten into a dimensionless form by choosing the reference values 
\begin{align*}
	&p_{ref} = 10^5\, N / m^2,\quad
	\rho_{ref} = 10\,kg/m^3, \quad
	t_{ref}   = 10^3 \, s, \quad
	l_{ref}    = 10^3 \, m, \\
	&U_{ref} = \frac{l_{ref}}{t_{ref}}=1\, m/s,\qquad  \theta_{ref} = \frac{p_{ref}}{R\rho_{ref}}=\frac{10^4}{R}\,K,
\end{align*}
so that the global Mach number is
\begin{equation*}
	\eps = \frac{U_{ref}}{\sqrt{p_{ref}/\rho_{ref}}} = 10^{-2}.
\end{equation*}
We run the simulation for a very long time until the final stopping time $t=700s$. 
We show the perturbation of potential temperature $\Delta \theta$ on a mesh gird of  $200\times200$ in Fig. ~{\ref{Fig6_2}}.
Inviscid wall boundary conditions \cite{wu2021uniformly} are used. We can see that both schemes can correctly capture the shear movement of the rising bubble, which is along the opposite direction of gravity, and eventually, it forms a shape of a mushroom cloud.
For this example with smooth solutions, we only use component-wise WENO reconstruction with linear weights for saving computational cost. However, comparing the WB-Xing scheme with our IMEX scheme, our scheme is much less diffusive on the same mesh grid, as the explicit WB-Xing scheme has a numerical viscosity inversely proportional to the Mach number $\eps$, while our IMEX scheme does not. We also compare the CPU cost of two schemes in Table~{\ref{table_2D_test3_1}}, from which we can observe that our IMEX scheme is much more efficient for this low Mach problem.


\begin{figure}[hbtp]
	\begin{center}
		{\includegraphics[width=5.5cm]{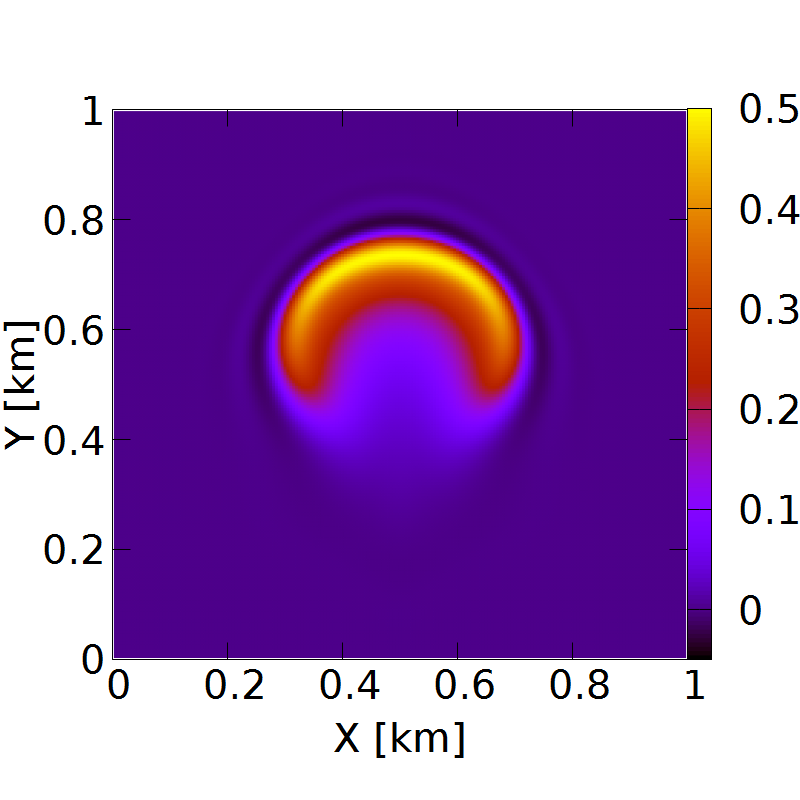}}
		{\includegraphics[width=5.5cm]{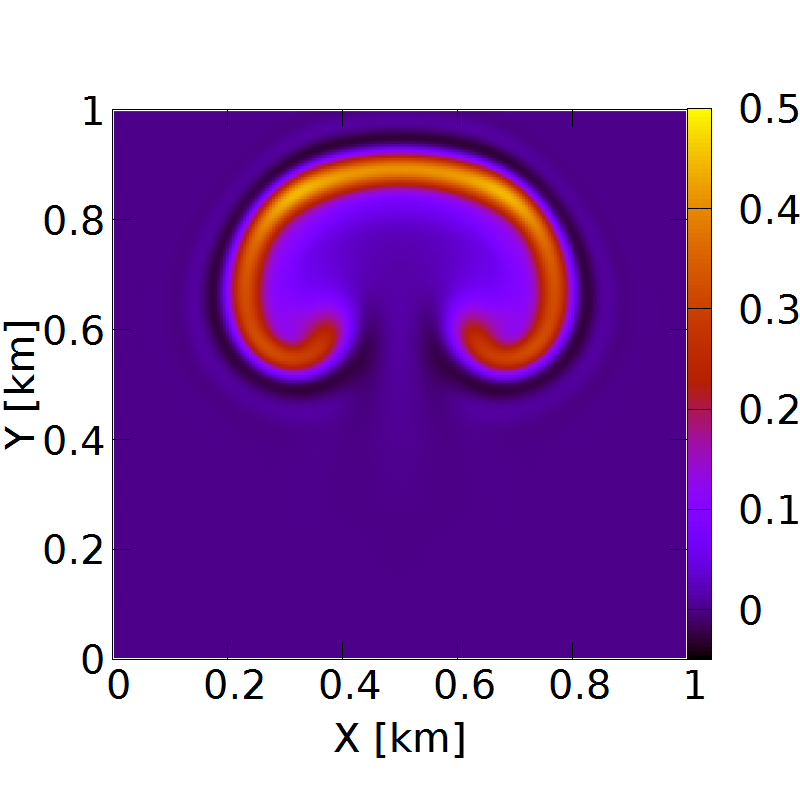}}
		{\includegraphics[width=5.5cm]{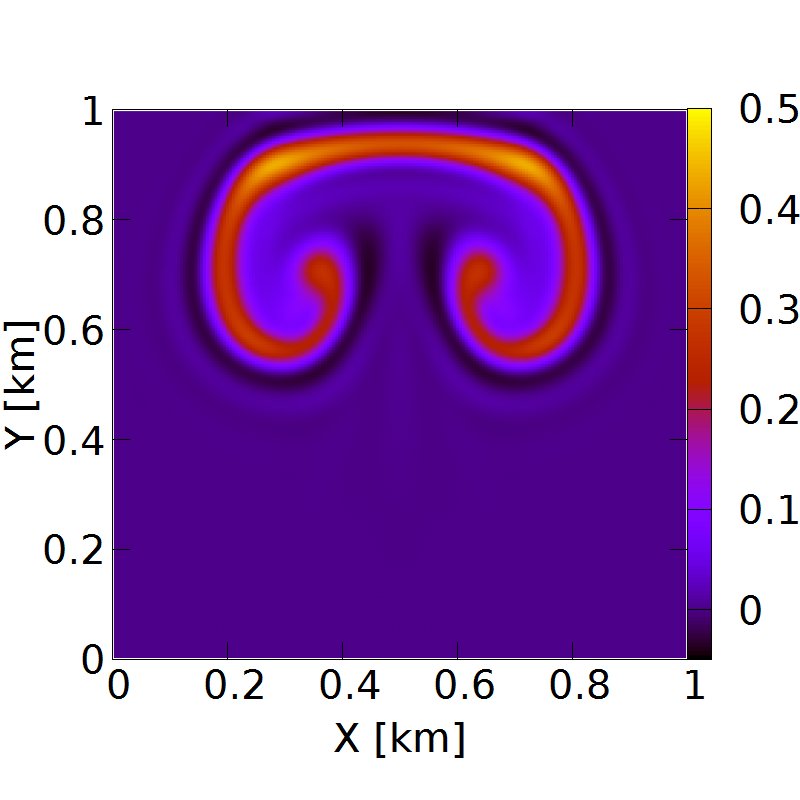}}
		{\includegraphics[width=5.5cm]{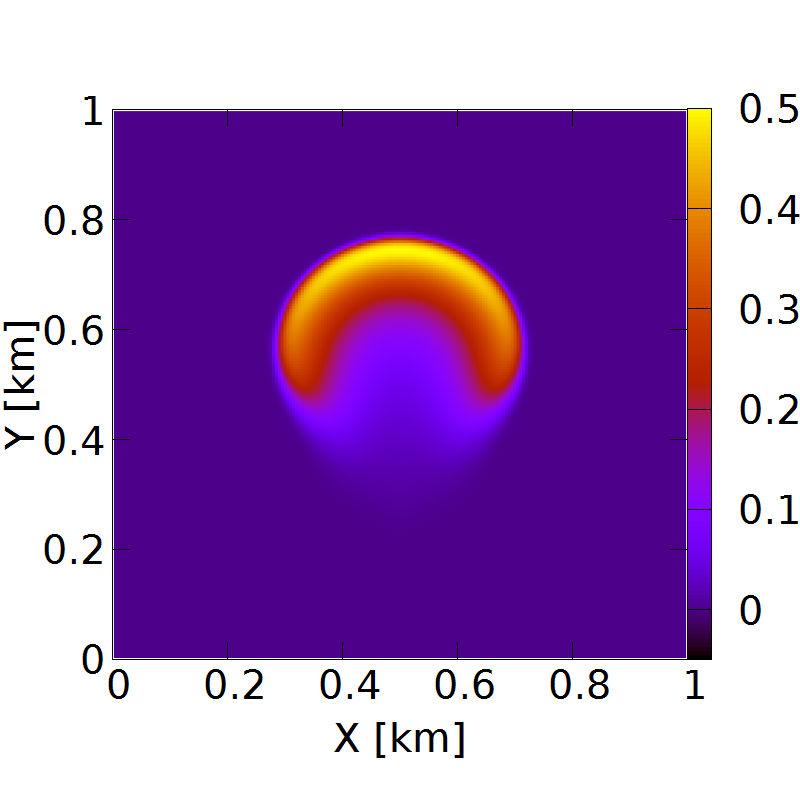}}
		{\includegraphics[width=5.5cm]{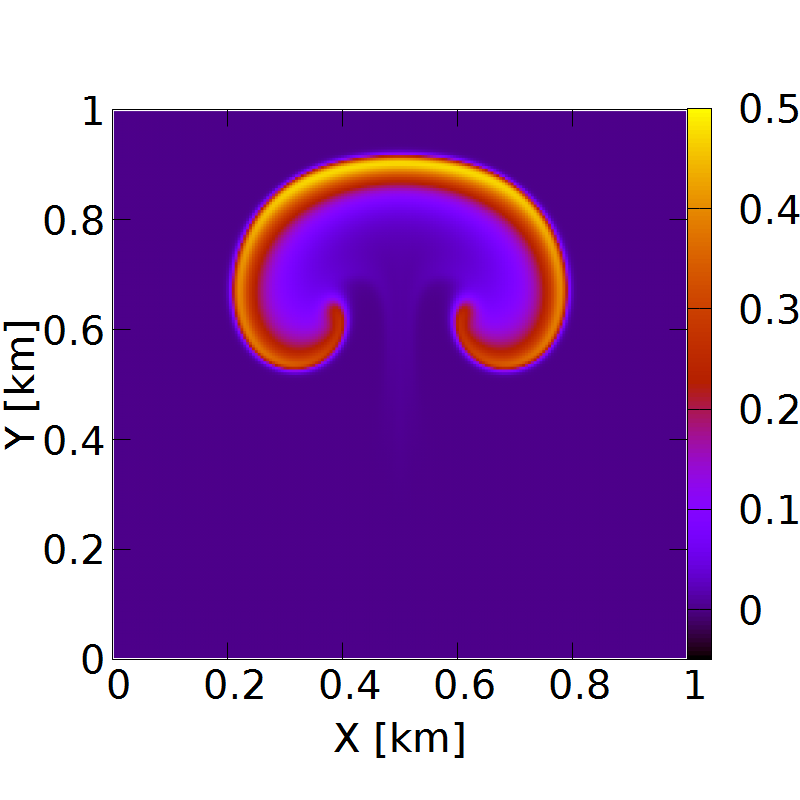}}
		{\includegraphics[width=5.5cm]{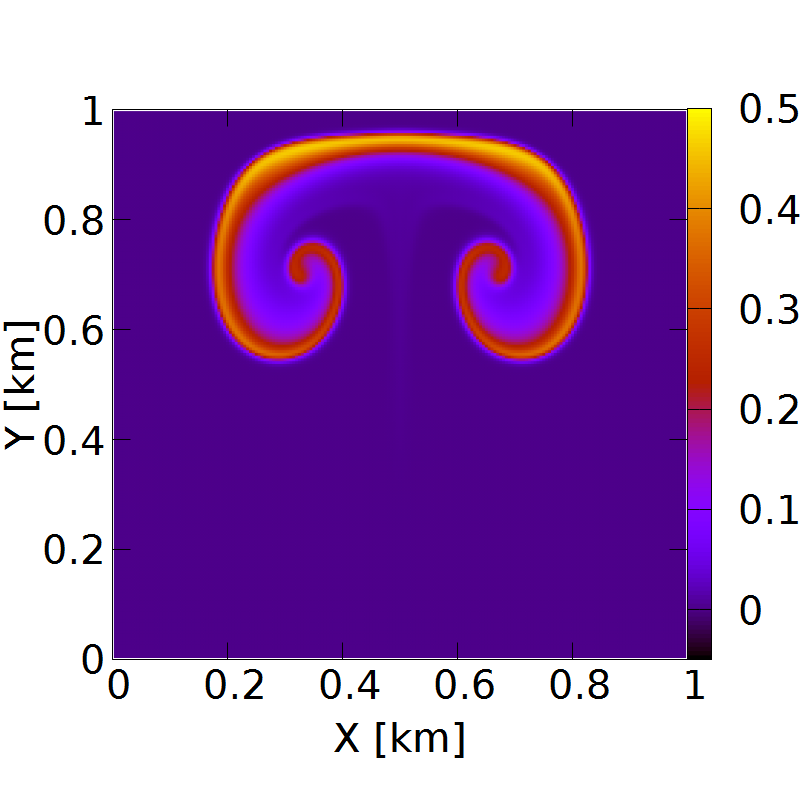}}
		\caption{ Example~\ref{exam6}. Numerical results for the perturbation of potential temperature $\Delta \theta$. From left to right, $t=400s, 600s, 700s$, respecively. Top: WB-Xing; Bottom: IMEX. Mesh: $200\times 200$.}
		\label{Fig6_2}
	\end{center}
\end{figure}


\subsection{Inertia-gravity wave}
\label{exam2D_test3}
This test arises from atmospheric flows, and has also been studied in \cite{ghosh2016well,giraldo2008study,wu2021uniformly}.
The computational domain is set as $[0,300000]\times[0,10000]\,m^2$, with a periodic boundary condition in the $x$ direction and an inviscid wall boundary condition in the $y$ direction.
The linear gravitational field with $(\Phi_x,\Phi_y) = (0,g)$ and $g=9.8\,m/s^2$ is considered.
The Exner pressure is
\begin{equation*}
	\Pi =1 +
	\frac{(\gamma - 1)g^2}{\gamma RT_0\mathscr{N}^2}
	\biggl[
	\exp
	\biggl(-
	\frac{\mathscr{N}^2}{g}
	y
	\biggr)-
	1
	\biggr],
\end{equation*}
and the initial potential temperature is denoted as 
\begin{equation*}
	\theta(x,y,0) = \theta_0(y) + \Delta\theta(x,y,0),
\end{equation*}
where 
\begin{equation*}
	\theta_0(y) = \bar{T} \exp
	\biggl(\frac{\mathscr{N}}{g}y \biggr)
	\,K,
	\quad
	\Delta \theta(x,y,0) =
	\theta_c\sin\left(\frac{\pi y}{h_c}\right)
	\biggl[1+(x-x_c)^2/a_c^2\biggr]^{-1},
\end{equation*}
with the Brunt--V\"ais\"al\"a frequency $\mathscr{N}  = 0.01/s$,
$\bar{T} = 300 \, K$, $\theta_c = 0.01\,K$, $h_c = 10000\,m$, $x_c = 100000\,m$ and $a_c = 5000\,m $.
The initial conditions are given by
\begin{equation*}
	\rho(x,y,0) = \frac{\bar{p}}{R \theta}\, \Pi^{\frac{1}{\gamma-1}},\quad
	\bu(x,y,0) = (20,0) \, m/s,\quad
	p(x,y,0) = \bar{p} \,\Pi^{\frac{\gamma}{\gamma-1}},
\end{equation*}
and the equilibrium state is
\begin{equation*}
	p_0 = \bar{p}\, \Pi^{\frac{\gamma}{\gamma-1}},\qquad
	\theta_0 = \bar{T} \exp
	\biggl(\frac{\mathscr{N}}{g}y \biggr)
	\,K,
	\qquad
	\rho_0 = \frac{\bar{p}}{R \theta_0} \Pi^{\frac{1}{\gamma-1}},
\end{equation*}
with $\bar{p}=10^5\, N/m^2 $.
As in Example~\ref{exam6}, we consider the problem in a dimensionless form, 
by choosing the following reference values
\begin{align*}
	&p_{ref} = 10^5 N / m^2,\quad
	\rho_{ref} = 10^{-1}kg/m^3, \quad
	t_{ref}    = 10^5m/s^2,\quad
	l_{ref}    = 10^5 m, \\
	&U_{ref} = \frac{l_{ref}}{t_{ref}}=1\, m/s,\qquad  \theta_{ref} = \frac{p_{ref}}{R\rho_{ref}}=\frac{10^6}{R}\,K,
\end{align*}
and the corresponding global Mach number is 
\begin{equation*}
	\eps = \frac{U_{ref}}{\sqrt{p_{ref}/\rho_{ref}}} = 10^{-3}.
\end{equation*}
In this example, for the WB-Xing scheme, an HLLC flux which is less diffusive than the Lax-Friedrichs flux, is used in the simulation. We run the solution up to $t=3000s$ for both schemes. For this example, we take CFL=$0.01$ for our IMEX scheme, and CFL=$0.005$ for the WB-Xing scheme.
We show the perturbation of potential temperature $\Delta \theta$ on the mesh grid $400\times50$ and $800\times50$ in Fig.~{\ref{Fig2D_test3_2}} and Fig.~{\ref{Fig2D_test3_1}}. From these numerical solutions, especially the cuts in Fig.~{\ref{Fig2D_test3_1}}, we can see the IMEX scheme has better resolution than the WB-Xing scheme on the coarse grid. We also compare the CPU cost of two schemes in Table~{\ref{table_2D_test3_1}}, and observe that the IMEX scheme is much more efficient in this low Mach regime.

\begin{figure}[htpb]
	\begin{center}			{\includegraphics[width=8cm]{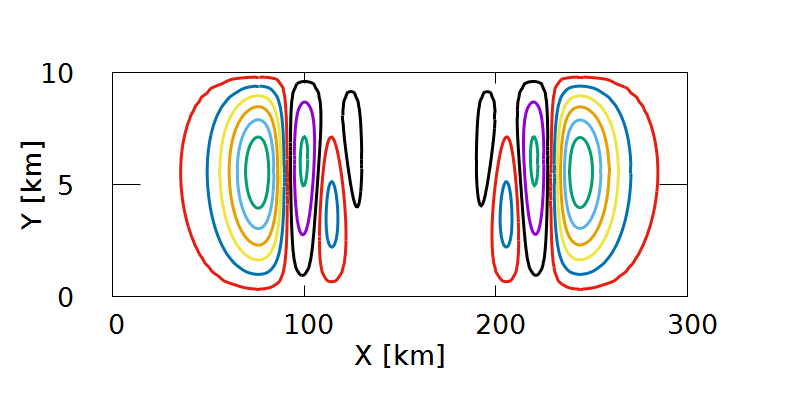}}				{\includegraphics[width=8cm]{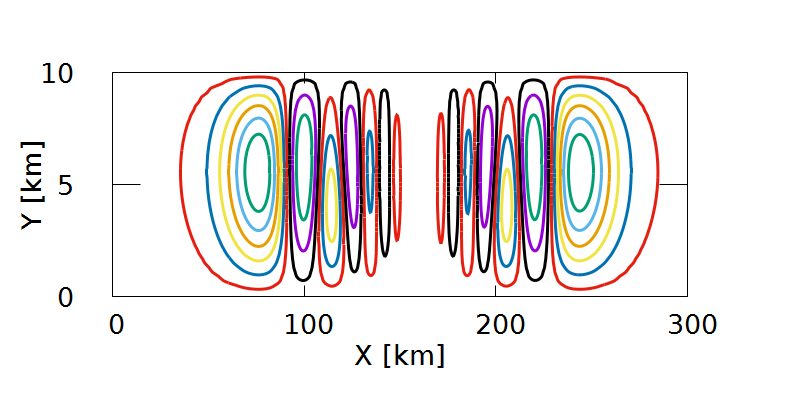}}
		{\includegraphics[width=8cm]{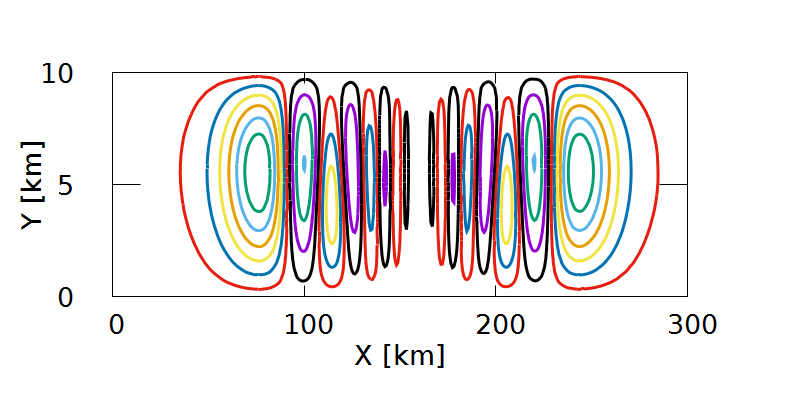}}				{\includegraphics[width=8cm]{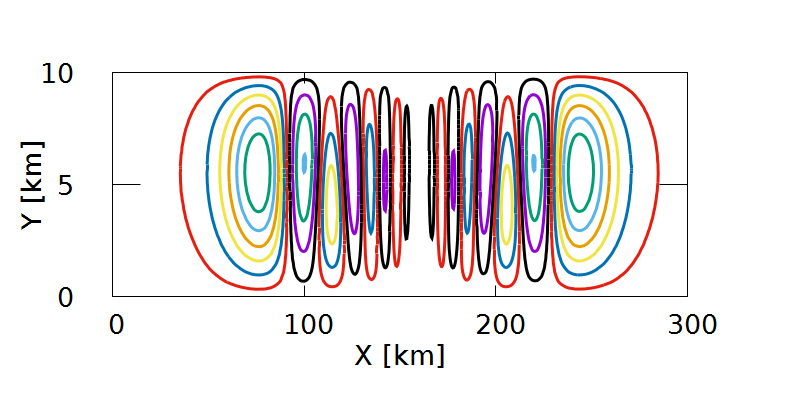}}
		\caption{  Example~\ref{exam2D_test3}. Numerical results for  the perturbation of potential temperature $\Delta \theta$ at $t=3000s$. Mesh grid: $400\times 50$ (left); $800\times 50$ (right). Top: WB-Xing; bottom: IMEX. }
		\label{Fig2D_test3_2}
	\end{center}
\end{figure}

\begin{figure}[htpb]
	\begin{center}
		{\includegraphics[width=8cm]{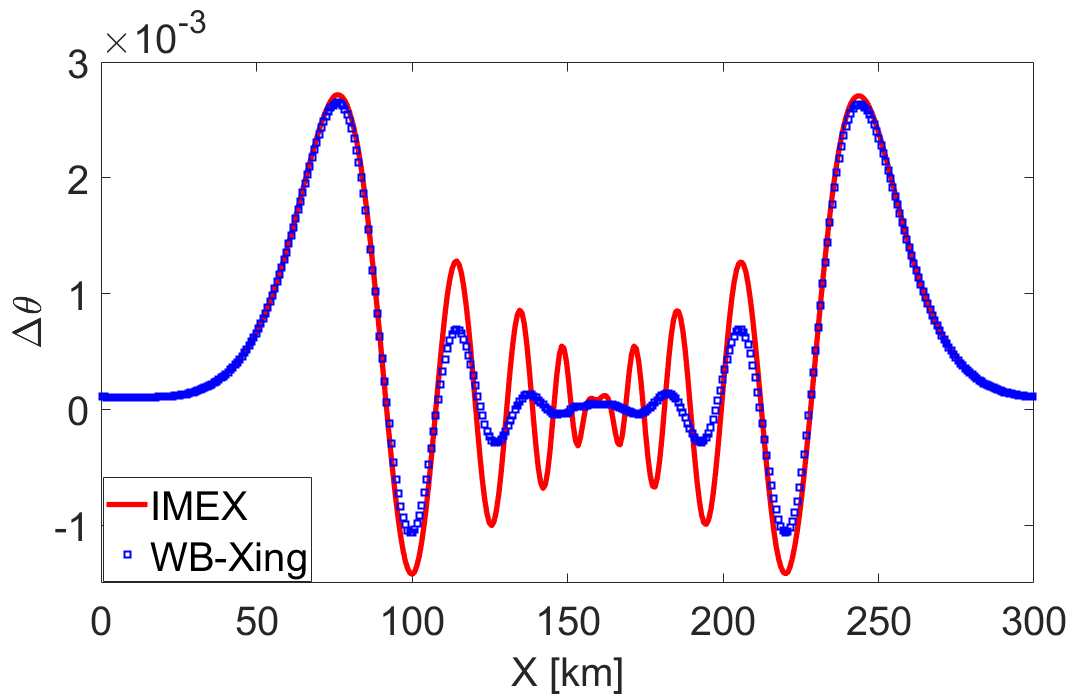}}
		{\includegraphics[width=8cm]{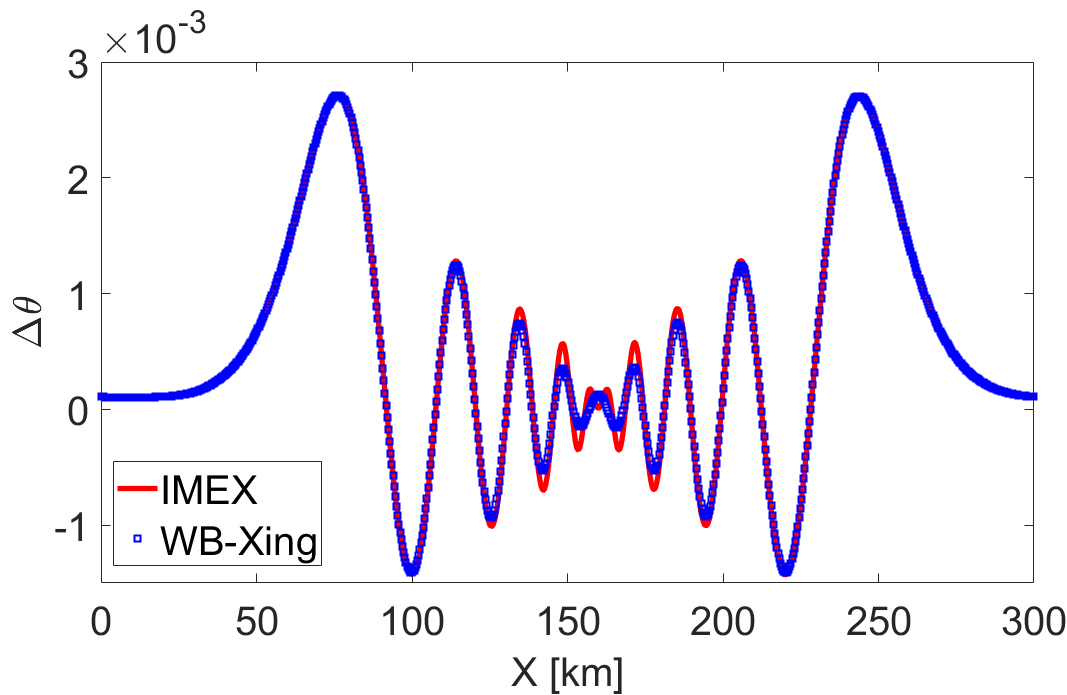}}
		\caption{ Example~\ref{exam2D_test3}. Numerical results for  the perturbation of potential temperature $\Delta \theta$ along the line $y=4900\,m$ at $t=3000s$. Mesh grid: $400\times 50$ (left); $800\times 50$ (right).}
		\label{Fig2D_test3_1}
	\end{center}
\end{figure}

\renewcommand{\multirowsetup}{\centering}
\begin{table}[htbp]
	\caption{ CPU cost (seconds) for the WB-Xing and IMEX schemes.}
	\begin{center}
		\begin{tabular}{c|c|c|c}
			\hline
			& $N_x \times N_y$     & WB-Xing& IMEX   \\ \hline
			Example~\ref{exam6} &
			$100\times 100$  &  7,328.77  &     1,354.02  \\ 
			$\eps=10^{-2}$ &
			$200\times 200$  &  55,719.46 &     11,712.07  \\\hline
			
			Example~\ref{exam2D_test3}&
			$400\times 50$  & 300,512.70         &     12,578.01 \\
			$\eps=10^{-3}$  &
			$800\times 50$  & 585,356.01        &     32,064.40 \\ \hline		
		\end{tabular}
	\end{center}
	\label{table_2D_test3_1}
\end{table}

\section{Conclusion}
\label{sec6}
\setcounter{equation}{0}
\setcounter{figure}{0}
\setcounter{table}{0}

In this work, we designed a high order semi-implicit AP well-balanced finite difference WENO scheme for the all Mach full Euler system with a gravitational field. It is much more challenging than the shallow water equations with a non-flat bottom \cite{huang2022high}, due to the existence of gravity which couples all equations together in the limit. We proposed to add the evolution for the perturbation of potential temperature in the design of our scheme, which provides a general and easy framework for the development of AP schemes to ensure the correct incompressible limit. The AP and AA properties are formally analyzed. Numerical tests on both 1D and 2D problems have demonstrated the high order accuracy, well-balanced, AP and AA properties of our proposed scheme. Our IMEX scheme has also been shown to yield better resolution and higher efficiency in the low Mach regime when compared with the explicit scheme.

\bibliographystyle{abbrv}
\bibliography{refer}

\end{document}